\documentclass[11pt,a4paper]{amsart}
\usepackage[foot]{amsaddr}
\usepackage{fullpage}

\usepackage{scalerel,amssymb}
\usepackage{amsthm,aliascnt}
\usepackage{nicefrac}
\usepackage{amsmath}

\makeatletter
\renewcommand{\email}[2][]{%
  \ifx\emails\@empty\relax\else{\g@addto@macro\emails{,\space}}\fi%
  \@ifnotempty{#1}{\g@addto@macro\emails{\textrm{(#1)}\space}}%
  \g@addto@macro\emails{#2}%
}
\makeatother

\usepackage[utf8]{inputenc}
\usepackage[T1]{fontenc}
\usepackage[british]{babel}

\usepackage[
backend=biber,
style=alphabetic,
doi=false,url=false,
maxbibnames=10 % as many authors as required
]{biblatex}
\renewbibmacro{in:}{}

\usepackage{xcolor}
\usepackage[breaklinks]{hyperref}
\usepackage[capitalize,nameinlink,noabbrev]{cleveref}
\usepackage{breakurl}

\definecolor{lgreen}{rgb}{0.0, 0.48, 0.0}
\definecolor{lpurple}{rgb}{0.48, 0.0, 0.48}
\definecolor{bblue}{rgb}{0.2, 0.4, 0.8}
\hypersetup{linktocpage,
            colorlinks=true,
            linkcolor=lgreen,
            citecolor=lpurple,
            linktoc=true}
\makeatletter
\renewcommand{\tocsection}[3]{%
  \indentlabel{\@ifnotempty{#2}{\bfseries\ignorespaces#1 #2\quad}}\bfseries#3}
\renewcommand{\tocsubsection}[3]{%
  \indentlabel{\@ifnotempty{#2}{\ignorespaces#1 #2\quad}}#3}
\makeatother

\usepackage{floatrow}
\usepackage{float}
\usepackage{algorithm,algorithmic}
\Crefname{ALC@unique}{Line}{Lines}

\usepackage{graphicx}
\usepackage{epstopdf}
\usepackage{float}
\graphicspath{{./images/}}
\usepackage{caption, subcaption}

\definecolor{blackred}{rgb}{0.6, 0.3, 0.3}
\usepackage{listings}

\definecolor{darkgray}{rgb}{.4,.4,.4}

\lstdefinelanguage{haskell}{
    keywords={Leaf, Unary, Motzkin, Binary},
    comment=[l]{--}
}

\usepackage{catoptions}
\makeatletter

\def\Autoref#1{%
  \begingroup
  \edef\reserved@a{\cpttrimspaces{#1}}%
  \ifcsndefTF{r@#1}{%
    \xaftercsname{\expandafter\testreftype\@fourthoffive}
      {r@\reserved@a}.\\{#1}%
  }{%
    \ref{#1}%
  }%
  \endgroup
}
\def\testreftype#1.#2\\#3{%
  \ifcsndefTF{#1autorefname}{%
    \def\reserved@a##1##2\@nil{%
      \uppercase{\def\ref@name{##1}}%
      \csn@edef{#1autorefname}{\ref@name##2}%
      \autoref{#3}%
    }%
    \reserved@a#1\@nil
  }{%
    \autoref{#3}%
  }%
}
\makeatother

\newcommand{\set}[1]{\{#1\}}
\newcommand{\seq}[1]{\left(#1\right)}
\newcommand{\idx}[1]{\mbox{\underline{\sf #1}}}
\def\vec{\boldsymbol}

\DeclareMathOperator*{\Bern}{\mbox{Bern}}

\DeclareMathOperator*{\Seq}{\mbox{\sc Seq}}
\DeclareMathOperator*{\Set}{\mbox{\sc Set}}
\DeclareMathOperator*{\MSet}{\mbox{\sc MSet}}

\DeclareMathOperator*{\Cycle}{\mbox{\sc Cyc}}
\DeclareMathOperator*{\Cov}{\mbox{\rm Cov}}

\newcommand{\CS}[1]{\mathcal{#1}}
\newcommand{\At}{t}

\usepackage{microtype}

\DeclareMathAlphabet\mathbfcal{OMS}{cmsy}{b}{n}

\newcommand{\mynewtheorem}[2]{
  \newaliascnt{#1}{dummy}
  \newtheorem{#1}[#1]{#2}
  \aliascntresetthe{#1}
  \expandafter\def\csname #1autorefname\endcsname{#2}
}

\theoremstyle{definition}
\mynewtheorem{theorem}{Theorem}
  \mynewtheorem{proposition}{Proposition}
  \mynewtheorem{corollary}{Corollary}
  \mynewtheorem{lemma}{Lemma}

\theoremstyle{definition}
  \mynewtheorem{Definition}{Definition}
  \mynewtheorem{Example}{Example}
  \mynewtheorem{Remark}{Remark}
  \mynewtheorem{condition}{Condition}

\begin{document}

\title{Polynomial tuning of multiparametric\\ combinatorial samplers}\thanks{
Maciej Bendkowski was partially supported within the Polish National Science
Center grant 2016/21/N/ST6/01032 and the French Government Scholarship within
the French-Polish POLONIUM grant number 34648/2016.  Olivier Bodini and Sergey
Dovgal were supported by the French project ANR project MetACOnc,
ANR-15-CE40-0014.
}
\author{Maciej Bendkowski${}^1$}
\address[1]{
      Theoretical Computer Science Department,
      Faculty of Mathematics and Computer Science,
      Jagiellonian University, {\L}ojasiewicza 6,
      30-348 Krak\'ow, Poland.
}
\email{bendkowski@tcs.uj.edu.pl}
\author{Olivier Bodini${}^2$}
\address[2]{
      Institut Galilée,
      Université Paris 13,
      99 Avenue Jean Baptiste Clément, 93430
      Villetaneuse, France.
}
\email{\{Olivier.Bodini, Dovgal\}@lipn.univ-paris13.fr}
\author{Sergey Dovgal${}^{2,3,4}$}
\address[3]{
      Institut de Recherche en Informatique Fondamentale,
      Université Paris 7,
      5 Rue Thomas Mann
      75013 Paris,
      France and (4)
      Moscow Institute of Physics and Technology,
      Institutskiy per. 9, Dolgoprudny, Russia 141700
}
\date{\today}

\maketitle

\begin{abstract} 
Boltzmann samplers and the recursive method are prominent algorithmic
    frameworks for the approximate-size and exact-size random generation of
    large combinatorial structures, such as maps, tilings, RNA sequences or
    various tree-like structures. In their multiparametric variants, these
    samplers allow to control the profile of expected values corresponding to
    multiple combinatorial parameters. One can control, for instance, the
    number of leaves, profile of node degrees in trees or the number of certain
    subpatterns in strings. However, such a flexible control requires an
    additional non-trivial tuning procedure. In this paper, we propose an
    efficient polynomial-time, with respect to the number of tuned parameters,
    tuning algorithm based on convex optimisation techniques.
    Finally, we illustrate the efficiency of our approach using
    several applications of rational, algebraic and Pólya structures including
    polyomino tilings with prescribed tile frequencies, planar
    trees with a given specific node degree distribution, and weighted
partitions.
\end{abstract}

\section{Introduction}
Uniform random generation of combinatorial structures forms a prominent
research area of computer science with multiple important applications ranging
from automated software testing techniques, see~\cite{Claessen-2000}, to
complex simulations of large physical statistical models,
see~\cite{BhaCouFahRan2017}. Given a formal specification defining a set of
combinatorial structures (for instance graphs, proteins or tree-like data
structures) we are interested in their efficient random sampling ensuring the
uniform distribution among all structures sharing the same size.

One of the earliest examples of a generic sampling template is Nijenhuis and
Wilf's recursive method~\cite{NijenhuisWilf1978} later systematised by
Flajolet, Zimmermann and Van Cutsem~\cite{FLAJOLET19941}. In this approach, the
generation scheme is split into two stages -- an initial preprocessing phase
where recursive branching probabilities dictating subsequent sampler decisions
are computed, and the proper sampling phase itself.  Alas, in both phases the
algorithm manipulates integers of size exponential in the target size $n$,
turning its effective bit complexity to $O(n^{3+\varepsilon})$, compared to
$\Theta(n^2)$ arithmetic operations required.  Denise and Zimmermann reduced
later the average-case bit complexity of the recursive method to $O(n \log n)$ in time
and $O(n)$ in space using a certified floating-point arithmetic
optimisation~\cite{DenZimm99}. Regardless, worst-case space bit complexity
remained $O(n^2)$ as well as bit complexity for non-algebraic languages.
Remarkably, for rational languages Bernardi and
Giménez~\cite{Bernardi2012} recently linked the floating-point optimisation of
Denise and Zimmermann with a specialised divide-and-conquer scheme reducing
further the worst-case space bit complexity and the average-case
time bit complexity to $O(n)$.

A somewhat relaxed, approximate-size setting of the initial generation problem
was investigated by Duchon, Flajolet, Louchard and Schaeffer who proposed a
universal sampler construction framework  of so-called Boltzmann
samplers~\cite{DuFlLoSc}. The key idea in their approach is to embed the
generation scheme into the symbolic method of analytic
combinatorics~\cite{flajolet09} and, in consequence, obtain an effective
recursive sampling template for a wide range of existing combinatorial classes.
In recent years, a series of important improvements was proposed for both
unlabelled and P\'{o}lya structures. Let us mention for instance linear
approximate-size (and quadratic exact-size) Boltzmann samplers for planar
graphs~\cite{fusy2005quadratic}, general-purpose  samplers for unlabelled
structures~\cite{flajolet2007boltzmann}, efficient samplers for plane
partitions~\cite{bodini2010random} or the cycle pointing operator for P\'{o}lya
structures~\cite{bodirsky2011boltzmann}. Moreover, the
framework was generalised onto differential
specifications~\cite{bodini2012boltzmann,bodini2016increasing};
linear exact-size samplers for Catalan and Motzkin trees were obtained, exploiting
the shape of their holonomic specifications~\cite{bacher2013exact}.

What was left open since the initial work of Duchon et al., was the development
of (i) efficient Boltzmann oracles providing effective means of evaluating
combinatorial systems within their disks of convergence and (ii) an automated
tuning procedure controlling the expected sizes of parameter values
of generated structures. The former problem was finally addressed by Pivoteau,
Salvy and Soria~\cite{PiSaSo12} who defined a rapidly converging combinatorial
variant of the Newton oracle by lifting the combinatorial version of Newton's
iteration of Bergeron, Labelle and Leroux~\cite{species} to a new numerical
level. In principle, using their Newton iteration and an appropriate use of
binary search, it became possible to approximate the singularity of a given
algebraic combinatorial system with arbitrarily high precision. However, even
if the singularity $\rho$ is estimated with precision $10^{-10}$ its
approximation quality does not correspond to an equally accurate approximation
of the generating function values at $\rho$, often not better than $10^{-2}$.
Precise evaluation at $z$ close to $\rho$ requires an extremely accurate
precision of $z$.  Fortunately, it is possible to trade-off the evaluation
precision for an additional rejection phase using the idea of analytic
samplers~\cite{BodLumRolin} retaining the uniformity even with rough evaluation
estimates.

Nonetheless, frequently in practical applications including for instance
semi-automated software testing techniques, additional control over the
internal structure of generated objects is required, see~\cite{palka2012}.
In~\cite{BodPonty} Bodini and Ponty proposed a multidimensional Boltzmann
sampler model, developing a tuning algorithm meant for the random generation of
words from context-free languages with a given target
letter frequency vector.  However, their algorithm
converges only in an \emph{a priori} unknown vicinity of the target
tuning variable vector.  In practice, it is therefore possible to control no
more than a few tuning parameters at the same time.

In the present paper we propose a novel polynomial-time tuning algorithm based
on convex optimisation techniques, overcoming the previous convergence
difficulties. We demonstrate the effectiveness of our approach with several
examples of rational, algebraic and P\'{o}lya structures. Remarkably, with our
new method, we are easily able to handle large combinatorial systems with
thousands of combinatorial classes and tuning parameters.

In order to illustrate the effectiveness of our approach, we have implemented a
prototype sampler generator {\sf Boltzmann Brain} ({\sf bb} in short).
The source code is available at
Github\footnote{\url{https://github.com/maciej-bendkowski/boltzmann-brain}}.
Supplementary scripts used to generate and visualise the
presented applications of this paper are available as a separate
repository\footnote{\url{https://github.com/maciej-bendkowski/multiparametric-combinatorial-samplers}}.

In \S~\ref{section:sampling:Boltzmann:principles} we briefly recall the
principles of Boltzmann sampling. Next, in \S~\ref{section:tuning} we describe
the tuning procedure. In \S~\ref{section:applications} we propose four
exemplary applications and explain the interface of {\sf bb}. Finally, in the appendix we give the proofs of the
theorems, discuss implementation details and describe a novel exact-size
sampling algorithm for strongly connected rational grammars.

\section{Sampling from Boltzmann principles}
\label{section:sampling:Boltzmann:principles}
\subsection{Specifiable $k$-parametric combinatorial classes.}
Let us consider the neutral class $\mathcal{E}$ and its
atomic counterpart $\mathcal{Z}$, both equipped with a
finite set of admissible operators
(i.e.~disjoint union $+$, Cartesian product $\times$, sequence $\Seq$,
multiset $\MSet$ and cycle $\Cycle$),
see~\cite[24--30]{flajolet09}.
Combinatorial specifications are finite systems of equations (possibly
recursive) built from elementary classes $\mathcal{E}$, $\mathcal{Z}$ and the
admissible operators.
\begin{Example}
Consider the following joint specification for $\mathcal{T}$
and $\mathcal{Q}$. In the combinatorial class \( \mathcal T \) of trees, nodes of even
level (the root starts at level one) have either no or two children  and each node at
odd level has an arbitrary number of non-planarily ordered children:
\begin{equation}
\label{eq:running:univariate}
\begin{cases}
    \mathcal{T}=\mathcal{Z} \MSet(\mathcal{Q})\ , \\
    \mathcal{Q}=\mathcal{Z}+\mathcal{Z}\mathcal{T}^2\, .
\end{cases}
\end{equation}
\end{Example}
In order to distinguish (in other words \emph{mark}) some additional combinatorial
parameters we consider the following natural multivariate extension of
specifiable classes.  \begin{Definition}{(Specifiable $k$\nobreakdash-parametric
    combinatorial classes)} A specifiable $k$\nobreakdash-parametric combinatorial class is
    a combinatorial specification built, in a possibly recursive manner, from
    $k$ distinct atomic classes $\mathcal{Z}_i$ (\( i \in \set{1, \ldots, k}
    \)), the neutral class $\mathcal{E}$ and admissible operators $+, \times,
    \Seq, \MSet$ and $\Cycle$.  In particular, a vector \( \vec{\mathcal C} =
    \seq{\mathcal{C}_1,\ldots,\mathcal{C}_m} \) forms a specifiable
    $k$\nobreakdash-parametric combinatorial class if its specification can be written down
    as
\begin{equation}
\label{eq:specifiable:class}
\begin{cases}
    \mathcal C_1
    =
    \Phi_1(\vec{\CS C}, \CS Z_1, \ldots, \CS Z_k)\, , \\
        \vdots  \\
    \mathcal C_m
    =
    \Phi_m(\vec{\CS C}, \CS Z_1, \ldots, \CS Z_k)
\end{cases}
\end{equation}
    where the right-hand side expressions
    are composed from $\vec{\CS C},\CS Z_1,\ldots,\CS Z_k$, admissible
    operators and the neutral class \( \mathcal{E} \). Moreover,
    we assume that specifiable $k$\nobreakdash-parametric combinatorial specifications
    form well-founded aperiodic systems, see~\cite{species,PiSaSo12,
    drmota1997systems}.
\end{Definition}

\begin{Example}
    \label{example:multivariate}
Let us continue our running example, see~\eqref{eq:running:univariate}.  Note
    that we can introduce two additional \emph{marking} classes \( \mathcal U \) and
    \( \mathcal V \) into the system, of weight zero each, turning it in effect
    to a $k$\nobreakdash-specifiable combinatorial class as follows:
\begin{equation}
\label{eq:running:multivariate}
\begin{cases}
    \mathcal{T}=\CS U \CS Z \MSet(\mathcal{Q}),\\
    \mathcal{Q}=\CS V \CS Z + \mathcal{Z}\mathcal{T}^2\, .
\end{cases}
\end{equation}
In this example, $\mathcal{U}$ is meant to mark the occurrences of nodes at odd
    levels, whereas $\mathcal{V}$ is meant to mark leaves at even levels. In
    effect, we \emph{decorate} the univariate specification with explicit
    information regarding the internal structural patterns of our interest.
\end{Example}

Much like in their univariate variants, $k$\nobreakdash-parametric combinatorial
specifications are naturally linked to ordinary multivariate generating
functions, see e.g~\cite{flajolet09}.

\begin{Definition}{(Multivariate generating functions)}
The multivariate ordinary generating function in variables
$z_1,\ldots,z_k$ associated to a specifiable $k$\nobreakdash-parametric combinatorial
    class $\mathcal{C}$ is defined as
\begin{equation}
    \displaystyle C(z_1,\ldots,z_k)
    =
    \sum_{p_1\geq 0, \ldots, p_k\geq 0}
    c_{\boldsymbol{p}}
        \boldsymbol{z}^{\boldsymbol{p}}
\end{equation}
where $c_{\boldsymbol{p}}=c_{p_1,\ldots,p_k}$ denotes the number of structures
    with $p_i$ atoms of type $\CS Z_i$ and $\boldsymbol{z}^{\boldsymbol{p}}$
    denotes the product $z_1^{p_1} \cdots z_k^{p_k}$. In the sequel, we call
    $\boldsymbol{p}$ the (composition) size of the structure.
\end{Definition}

In this setting, we can easily lift the usual univariate generating function
building rules to the realm of multivariate generating functions associated to
specifiable $k$\nobreakdash-parametric combinatorial classes.
\autoref{table:constructions} summarises these rules.

\renewcommand{\arraystretch}{1.3}
\begin{table*}[tbp]
    \begin{equation*}
\begin{array}{r|l|l|l}
     \text{Class}
     &
     \text{Description}
     &
     C(\boldsymbol{z})
     &
     \Gamma \mathcal{C} ( \boldsymbol{z} )
     \\
\hline
\hline
     \text{Neutral}
     &
     \CS{C} = \{\varepsilon\}
     &
     C(\boldsymbol{z}) = 1
     &
     \varepsilon
     \\
\hline
     \text{Atom}
     &
     \CS{C} = \{\At_i\}
     &
     C(\boldsymbol{z}) = z_i
     &
     \square_i\\
\hline
     \text{Union}
     &
     \CS{C} = \CS{A} + \CS{B}
     &
     A(\boldsymbol{z}) +B(\boldsymbol{z})
     &
     \Bern\big(
       \frac{A(\boldsymbol{z})}{C(\boldsymbol{z})},
       \frac{B(\boldsymbol{z})}{C(\boldsymbol{z})}
     \big)
       \longrightarrow
     \Gamma \mathcal{A}(\boldsymbol{z})
     \;|\;
     \Gamma \mathcal{B}(\boldsymbol{z})
     \\
     \hline
     \text{Product}
     &
     \CS{C} = \CS{A} \times \CS{B}
     &
     A(\boldsymbol{z})
       \times
     B(\boldsymbol{z})
     &
     (
       \Gamma \mathcal{A}(\boldsymbol{z}),
       \Gamma \mathcal{B}(\boldsymbol{z})
     )
     \\
\hline
      \mbox{Sequence}
      &
      \CS{C} = \Seq(\CS{A})
      &
      (1-A(\boldsymbol{z}))^{-1}
      &
      \ell := \text{Geom}(1 - A(\boldsymbol{z})) \longrightarrow
      ( \Gamma \mathcal A(\vec z) )_{\times\ell \text{ times}}
      \\
\hline
      \text{MultiSet}
      &
      \MSet(\mathcal{A})
      &
      \exp\left(
          \sum_{m=1}^\infty
          \tfrac{1}{m}A(\boldsymbol{z}^m)
      \right)
      &
      \text{
          see \autoref{algorithm:mset}, \autoref{section:polya:structres}
      }
      \\
\hline
   \text{Cycle}
   &
   \Cycle(\mathcal{A})
   &
   \sum_{m=1}^\infty \!\!
   \frac{\varphi(m)}{m}
   \ln \frac{1}{1-A(\boldsymbol{z}^m)}
   &
   \text{
       see \autoref{algorithm:cycle}, \autoref{section:polya:structres}
   }
   \\
\end{array}
\end{equation*}
\caption{Multivariate generating functions and their Boltzmann samplers
$\Gamma\mathcal{C}(\boldsymbol{z})$. }\label{table:constructions}
\end{table*}

\subsection{Multiparametric Boltzmann samplers.}
Consider a typical multiparametric Boltzmann sampler workflow~\cite{BodPonty}
on our running example, see~\eqref{eq:running:multivariate}. We start with
choosing target expectation quantities \( (n, k, m) \) of nodes from atomic
classes \( (\mathcal Z, \mathcal U, \mathcal V) \). Next, using a dedicated
tuning procedure we obtain a vector of three real positive numbers \( \vec z =
(z, u, v) \) depending on \( (n,k,m) \). Then, we construct a set of
recursive Boltzmann samplers \( \Gamma \mathcal U(\vec z), \Gamma
\MSet(\mathcal Q(\vec z)) \), etc. according to the building rules
in~\autoref{table:constructions}.  Finally, we use the so constructed samplers
to generate structures with tuned parameters.

In order to sample from either $\CS E$ or atomic classes, we simply construct
the neutral element $\varepsilon$ or an appropriate atomic structure $\Box_i$,
respectively. For union classes we make a Bernoulli choice depending on the
quotients of respective generating functions values and continue with sampling
from the resulting class. In the case of product classes, we spawn two
independent samplers, one for each class, and return a pair of built
structures. Finally, for $\Seq(\CS A)$ we draw a random value from a geometric
distribution with parameter $1 - A(\boldsymbol{z})$ and spawn that many
samplers corresponding to the class $\CS A$.
In other words, \( \mathbb P(\ell \text{ instances}) = A(\vec z)^\ell (1 - A(\vec z)) \).
In the end, we collect the sampler
outcomes and return their list. The more involved $\MSet$ and $\Cycle$
constructions are detailed in~\autoref{section:polya:structres}.

The probability space associated to so constructed Boltzmann samplers takes then
the following form.  Let $\boldsymbol{z} \in (\mathbb{R}^+)^k$ be a vector
inside the ball of convergence of $C(\boldsymbol{z})$ and $\omega$ be a
structure of composition size $\boldsymbol{p}$ in a $k$\nobreakdash-parametric class
$\mathcal{C}$. Then, the probability that $\omega$ becomes the output of a
multiparametric Boltzmann sampler \( \Gamma \mathcal C(\vec z)\) is given as
\begin{equation}
\label{eq:WeightedBoltzmannDistribution}
    \mathbb{P}_{\boldsymbol{z}}(\omega)
    =
    \frac
        { {\vec z}^{\boldsymbol{p}} }
        { C ( \boldsymbol{z} ) }
    \enspace .
\end{equation}

\begin{proposition}
\label{proposition:expected:value}
Let \( \vec N = (N_1, \ldots, N_k) \) be the random vector where \( N_i \)
equals the number of atoms of type
    $\CS Z_i$ in a random combinatorial structure returned by the
    $k$\nobreakdash-parametric Boltzmann sampler $\Gamma\mathcal{C}(\vec z)$.
    Then, the expectation vector \( \mathbb E_{\vec z} (\vec N) \)
    and
 the covariance matrix
    $\Cov_{\boldsymbol{z}}(\boldsymbol{N})$ are given by
\begin{equation}
  \label{eq:expectations}
    \mathbb{E}_{\boldsymbol{z}}(N_i)
    =
    \left.
    \dfrac
        {\partial}
        {\partial \xi_i}
    \log C(e^{\vec \xi})
    \right|_{\vec \xi = \log \vec z}
    \quad \text{and} \quad
    \mathrm{Cov}_{\boldsymbol{z}}(\boldsymbol{N})
    \nonumber =
    \left.
    \left[
        \dfrac
            {\partial^2}
            {\partial \xi_i \partial \xi_j}
        \log C(e^{\vec \xi})
    \right]_{i,j = 1}^k
    \right|_{\vec \xi = \log \vec z}
    \enspace .
\end{equation}
Hereafter, we use \( e^{\boldsymbol z} \) to denote coordinatewise exponentiation.
\end{proposition}

\begin{corollary}
\label{corrolary:convexity}
The function \( \gamma(\vec z) := \log C(e^{\vec z}) \) is convex because its
    matrix of second derivatives, as a covariance matrix, is positive
    semi-definite inside the set of convergence.  This crucial assertion will
    later prove central to the design of our tuning algorithm.
\end{corollary}

\begin{Remark}
    Uniparametric recursive samplers of Nijenhuis and Wilf take, as well as
    Boltzmann samplers, a system of generating functions as their
    input.
    This system can be modified by
    putting fixed values of tuning variables,
    in effect altering the corresponding branching probabilities.
    The resulting distribution of the random variable corresponding to a
    weighted recursive sampler coincides with the distribution of the
    Boltzmann-generated variable conditioned on the structure size.  As a
    corollary, the tuning procedure that we discuss in the following section is
    also valid for the exact-size approximate-frequency recursive sampling.
    In~\autoref{section:rational:grammars} we
    describe an algorithm for rational specifications which samples
    objects of size \( n + O(1) \). As a by-product, we show how to
    convert approximate-size samplers corresponding to rational systems into
    exact-size samplers.
\end{Remark}

\section{Tuning as a convex optimisation problem} \label{section:tuning}
We start with a general result about converting the problem of tuning arbitrary
specifiable $k$\nobreakdash-parametric combinatorial specifications into a convex
optimisation problem, provided that one has access to an oracle yielding values
and derivatives of corresponding generating functions.  We note that this
general technique can be applied to differential specifications as well.  We
write \( f(\cdot) \to \min_{\boldsymbol z} \), \( f(\cdot) \to
\max_{\boldsymbol z} \) to denote the minimisation (maximisation, respectively)
problem of the target function \( f(\cdot) \) with respect to the vector
variable \( \vec z \).  All proofs are postponed
until~\autoref{section:convex:proofs}.  Throughout this section, we assume that
given tuning expectations are \emph{admissible} in the sense that there always
exists a target vector \( \vec z^\ast \) corresponding
to~\eqref{eq:expectations}. Furthermore, we assume that the combinatorial
system is \emph{well-founded} and \emph{strongly connected}. Some non-strongly
connected cases fall into the scope of our framework as well, but for the core
proof ideas we concentrate only on strongly connected systems.

\begin{theorem}
    \label{theorem:general}
Consider a multiparametric combinatorial class \( \mathcal C \).
Fix the expectations \( \mathbb E_{\vec z} \vec N = \vec \nu \),
see~Proposition~\ref{proposition:expected:value}.
    Let
\( C(\vec z) \) be the generating function corresponding to \( \mathcal C \).
    Then, the tuning vector \( \vec z \), see~\eqref{eq:expectations}, is
    equal to \( e^{\vec \xi} \) where \( \vec \xi \)
comes from the following minimisation problem:
\begin{equation}
    \log C(e^{\boldsymbol \xi}) - \boldsymbol\nu^\top \boldsymbol
\xi \to \min_{\boldsymbol \xi}
    \enspace .
\end{equation}
\end{theorem}

Let us turn to the specific classes of algebraic and rational specification. In
those cases, no differential-equation type systems are allowed;
however, it is possible to reformulate the problem so that no extra oracles
are required.

\begin{theorem}
\label{theorem:algebraic:tuning}
Let \( \vec{\mathcal C} = \vec \Phi(\vec{\mathcal C}, \vec{\mathcal Z}) \)
    be a multiparametric algebraic system with \( \vec{\CS C} =
    \seq{\CS C_1,\ldots,\CS C_m} \).
Fix the expectations \( N_i \) of the parameters of objects sampled from
\( \mathcal C_1 \) to
\( \mathbb E_{\vec z} \boldsymbol N = \boldsymbol \nu \).
Then, the tuning vector \( \boldsymbol z \) is equal to \( e^{\vec \xi} \) where \( \vec \xi \) comes from
the convex problem:
\begin{equation}
    \label{eq:optimisation:algebraic}
\begin{cases}
    c_1 - \boldsymbol \nu^\top \boldsymbol \xi \to \min_{\vec \xi, \vec c}
\enspace , \\
    \log \vec \Phi(e^{\vec c}, e^{\vec \xi}) - \vec c \leq 0 .
\end{cases}
\end{equation}
Hereafter, ``\( \leq \)'' and \( \log \vec \Phi
\) denote a set of inequalities and the coordinatewise
    logarithm, respectively.
\end{theorem}

Let us note that the above theorem naturally extends to the case of labelled
structures with \( \Set \) and \( \Cycle \) operators. For unlabelled P\'{o}lya
operators like \( \MSet \) or \(\Cycle \), we have to truncate the
specification to bound the number of substitutions.
In consequence, it becomes possible to sample corresponding unlabelled
structures, including partitions, functional graphs, series-parallel circuits,
etc.

 Singular Boltzmann samplers (also defined in~\cite{DuFlLoSc}) are the limiting
variant of ordinary Boltzmann samplers with an infinite expected size of
generated structures.  In their multivariate version, samplers are considered
\emph{singular} if their corresponding variable vectors belong to the boundary
of the respective convergence sets.
\begin{theorem}
\label{theorem:singular:tuning}
    Let \( \vec{\mathcal C} = \vec \Phi(\vec{\mathcal C}, \mathcal{Z}, \vec{\mathcal U}) \)
    be a multiparametric algebraic system with \( \vec{\CS C} =
    \seq{\CS C_1,\ldots,\CS C_m} \), the atomic class \(\mathcal{Z}\) marking
    the corresponding structure size and \( \vec{\CS U} =
    \seq{\CS U_1,\ldots,\CS U_k} \) being a vector (possibly empty) of distinguished atoms.
Assume that the target expected frequencies of the atoms \( \mathcal U_i \) are
    given by the vector
\( \boldsymbol \alpha \). Then, the variables \( (z, \vec u) \) that deliver
    the tuning of the corresponding singular Boltzmann sampler are the result
    of the following convex optimisation problem,
where \( z = e^{\xi} \), \( \vec u = e^{\vec \eta} \):
\begin{equation}
\begin{cases}
\xi + \vec \alpha^\top \vec \eta \to \max_{\xi, \vec \eta, \vec c} \enspace , \\
    \log \vec \Phi(e^{\vec c}, e^{\xi}, e^{\vec \eta}) - \vec c \leq 0\, .
\end{cases}
\end{equation}
\end{theorem}

Finally, let us note that all of the above outlined convex programs can be
effectively optimised using the polynomial-time interior-point method
optimisation procedure of Nesterov and Nemirovskii \cite{nesterov1994interior}.
The required precision \( \varepsilon \) is typically \( Poly(n) \),
see~\autoref{section:convex:proofs}.

\begin{theorem}
\label{theorem:main}
    For multiparametric combinatorial systems with description length \( L \),
    the tuning problem can be solved with precision \(
    \varepsilon \) in time \( O\big( L^{3.5}
    \log\frac{1}{\varepsilon}\big) \).
\end{theorem}

Let us complete this section by constructing an optimisation system
for~\eqref{example:multivariate}.  Let \( (n,k,m) \) be the target expectation
quantities of \( (\CS Z, \CS U, \CS V)\).
By the rules in~\autoref{table:constructions},
the system of
functional equations and its log-exp transformed optimisation counterpart take
the form
\begin{equation}
    \begin{cases}
        T(z, u, v) =
            u z \exp \left(
                \displaystyle\sum_{i=1}^{\infty} \dfrac{Q(z^i,u^i,v^i)}{i}
            \right) ,\\
        Q(z, u, v) =
            v z + z T(z, u, v)^2
            \, .
    \end{cases}
\end{equation}
Setting \( T(z^i,u^i,v^i)=e^{\tau_i} \), \( Q(z^i,u^i,v^i) =e^{\kappa_i} \),
\( z = e^{\zeta} \), \( u = e^{\eta} \), \( v = e^{\phi} \), we obtain
\begin{equation}
    \begin{cases}
         \tau_1 - n\zeta - k\eta - m\phi \to \min, \\
         \tau_j \geq \eta j + \zeta j + \displaystyle\sum_{i=1}^\infty
         \dfrac{e^{\kappa_{i j}}}{i}
         , \quad j \in  \{ 1, 2, \ldots \}\\
         \kappa_j \geq \log(e^{\phi j + \zeta j} + e^{\zeta j + 2 \tau_j})
         , \quad j \in  \{ 1, 2, \ldots \}
         \enspace .
    \end{cases}
\end{equation}
For practical purposes, the sum can be truncated with little effect on
distribution.

\section{Applications}
\label{section:applications}
In this section we present several examples illustrating the wide range of
applications of our tuning techniques. Afterwards, we briefly discuss our
prototype sampler generator and its implementation details.

\subsection{Polyomino tilings.}
We start with a benchmark example of a rational specification defining ${n
\times 7}$ rectangular tilings using up to $126$ different tile variants (a toy
example of so-called transfer matrix models, cf.~\cite[Chapter V.6, Transfer
matrix models]{flajolet09}).

\begin{figure}[hbt!]
    \begin{center}
        \includegraphics[height=0.02\textheight]{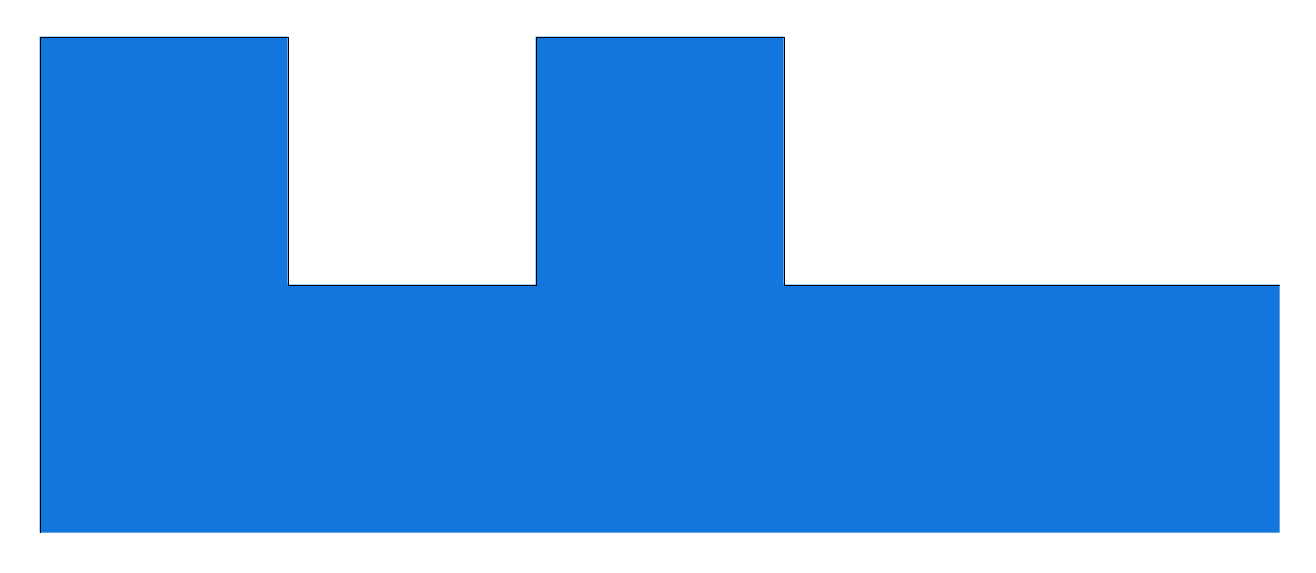} $\ $
        \includegraphics[height=0.02\textheight]{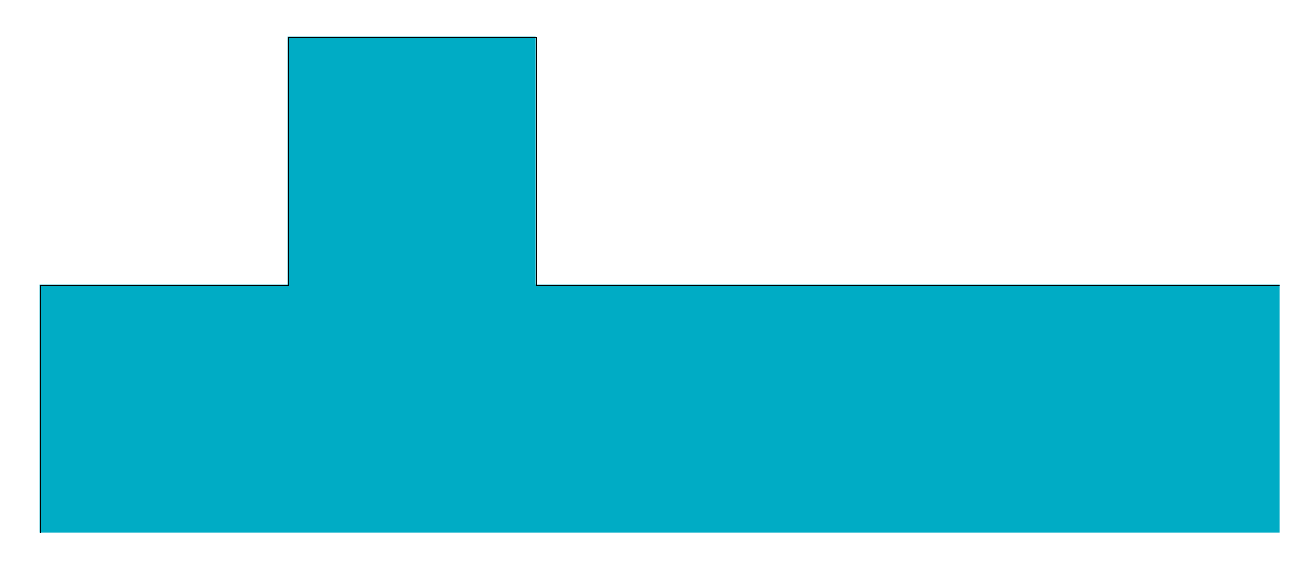} $\ $
        \includegraphics[height=0.02\textheight]{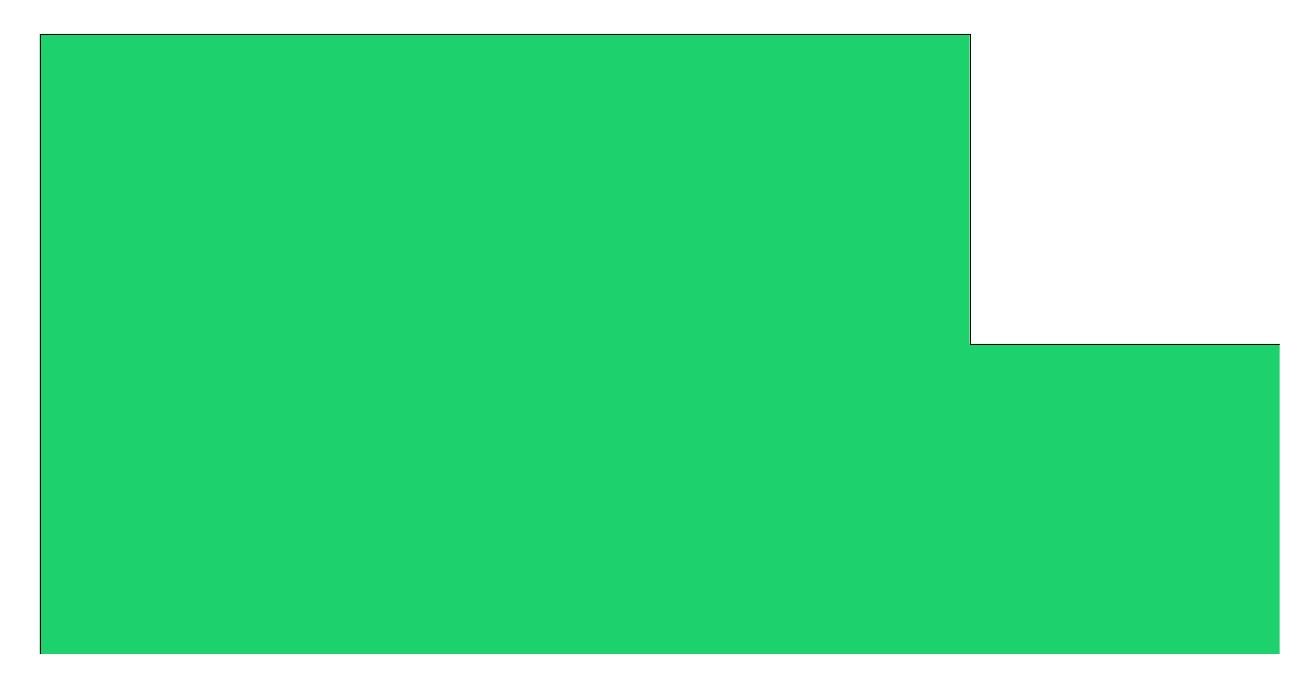} $\ $
        \includegraphics[height=0.02\textheight]{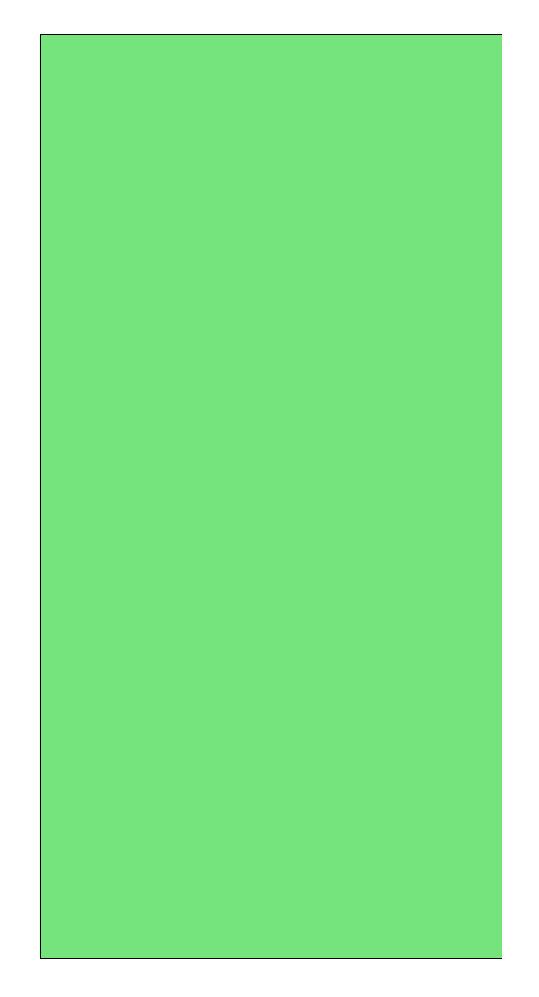} $\ $
        \includegraphics[height=0.02\textheight]{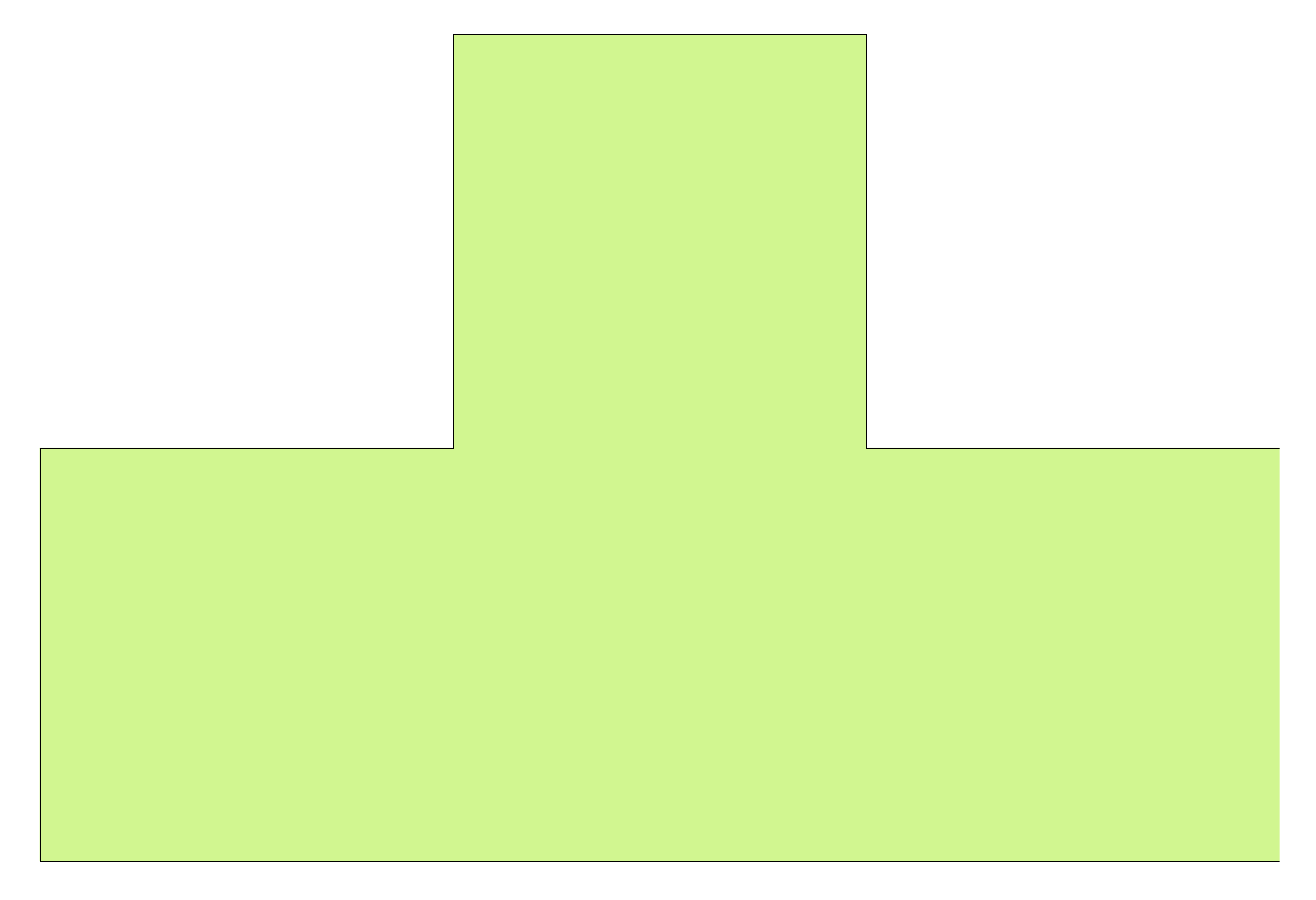} $\ $
        \includegraphics[height=0.02\textheight]{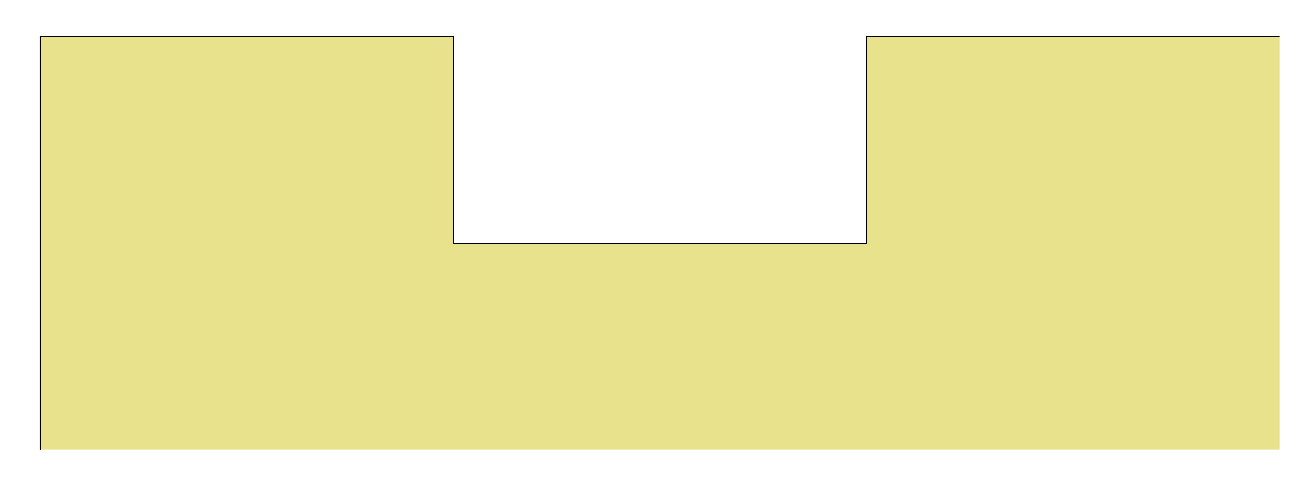} $\ $
        \includegraphics[height=0.02\textheight]{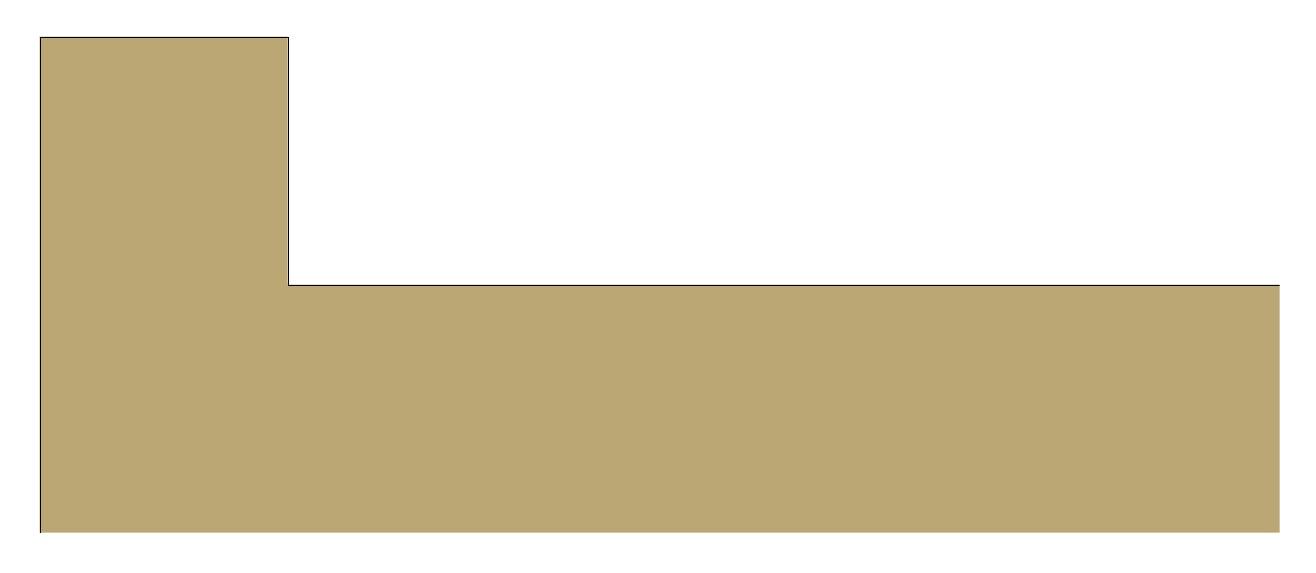} $\ $
        \includegraphics[height=0.02\textheight]{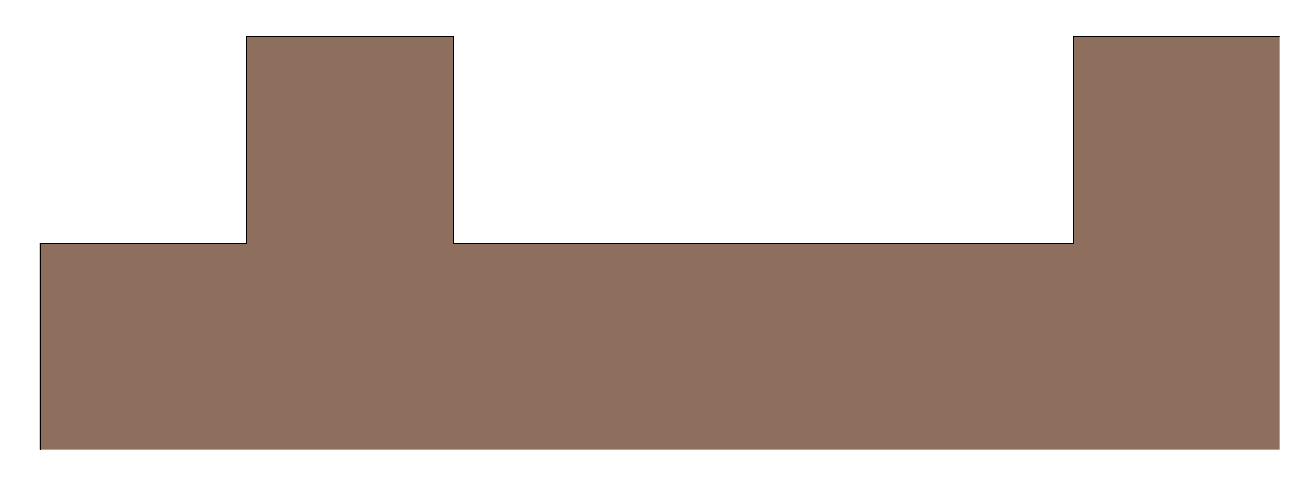} $\ $
        \includegraphics[height=0.02\textheight]{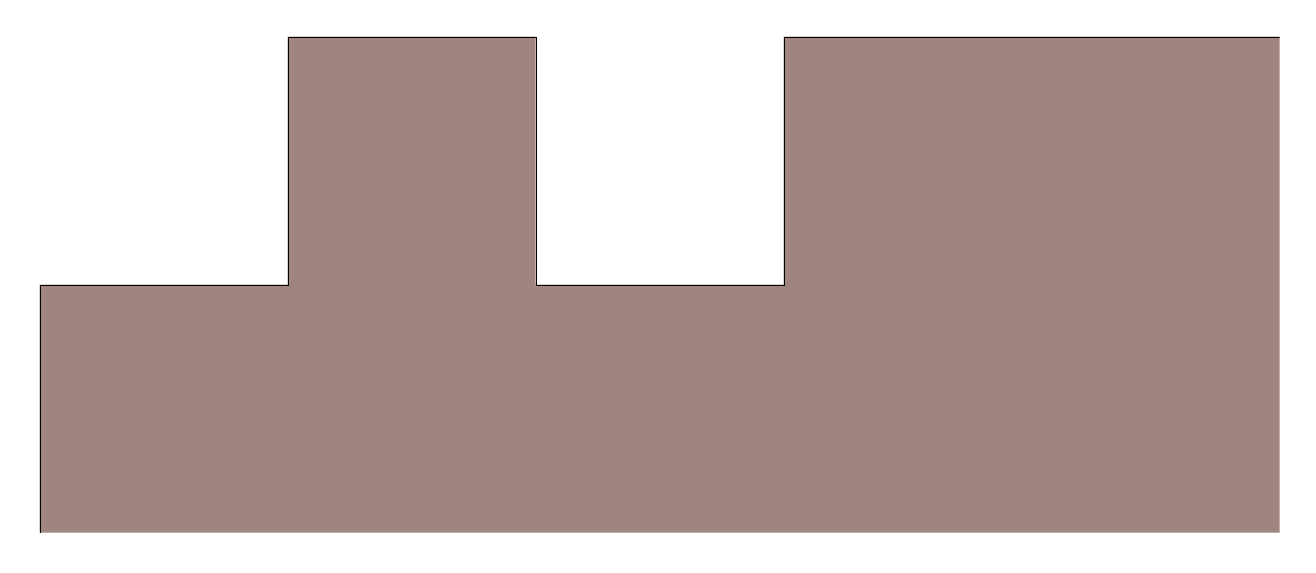} $\ $
        \includegraphics[height=0.02\textheight]{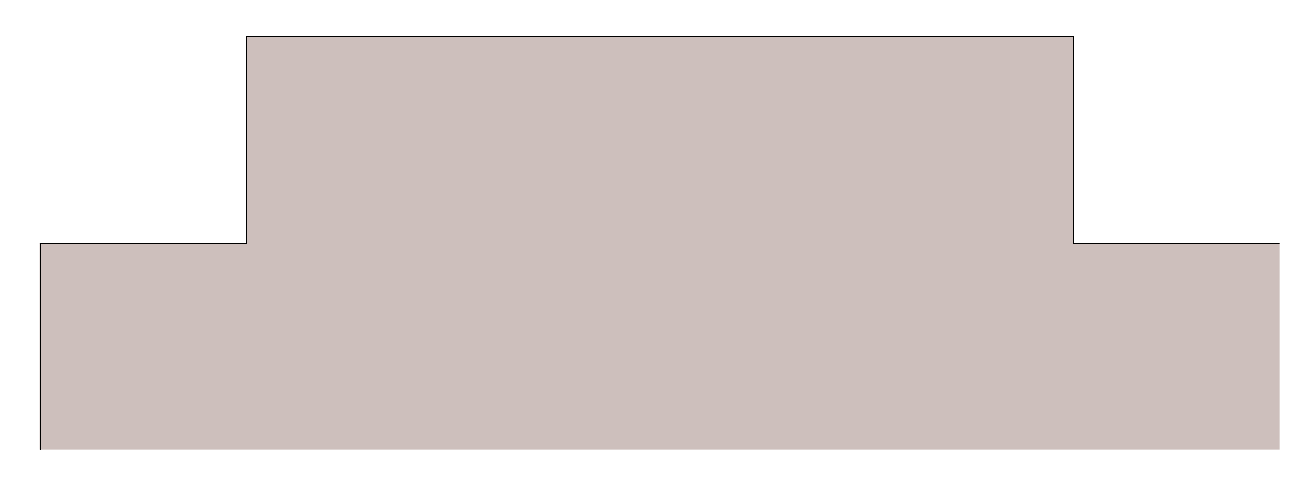}
    \end{center}
    \caption{Examples of admissible tiles}
    \label{fig:admissible:tiles}
\end{figure}

We begin the construction with defining the set $T$ of admissible tiles.  Each
tile $t \in T$ consists of two horizontal layers. The base layer is a single
connected block of width $w_t \leq 6$.  The second layer, placed on top of the
base one, is a subset (possibly empty) of $w_t$ blocks,
see \autoref{fig:admissible:tiles}.  For presentation purposes each
tile is given a unique, distinguishable colour.

Next, we construct the asserted rational specification following the general
construction method of defining a deterministic automaton with one state per
each possible partial tiling configuration using the set $T$ of available
tiles.  Tracking the evolution of attainable configurations while new tiles
arrive, we connect relevant configurations by suitable transition rules in the
automaton.  Finally, we (partially) minimise the constructed automaton removing
states unreachable from the initial empty configuration.  Once the automaton is
created, we tune the tiling sampler such that the target colour frequencies are
uniform, i.e.~each colour occupies, on average, approximately $\tfrac{1}{126}
\approx 0.7936\%$ of the outcome tiling area. \autoref{fig:tilings-7-126}
depicts an exemplary tiling generated by our sampler.
\begin{figure}[ht!]
\begin{subfigure}{.11\textwidth}
\centering
  \includegraphics[scale=0.8]{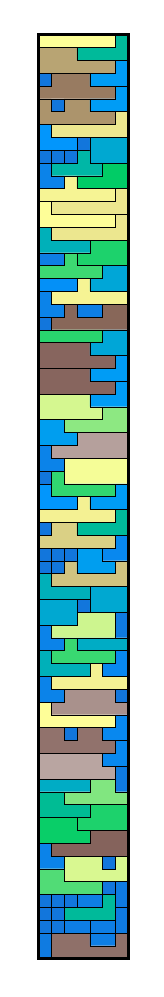}
\end{subfigure}                                       
\begin{subfigure}{.11\textwidth}                      
\centering                                            
  \includegraphics[scale=0.8]{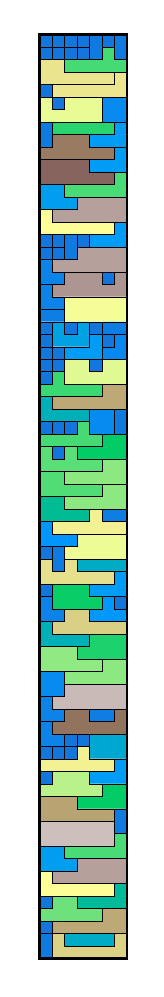}
\end{subfigure}                                       
\begin{subfigure}{.11\textwidth}                      
\centering                                            
  \includegraphics[scale=0.8]{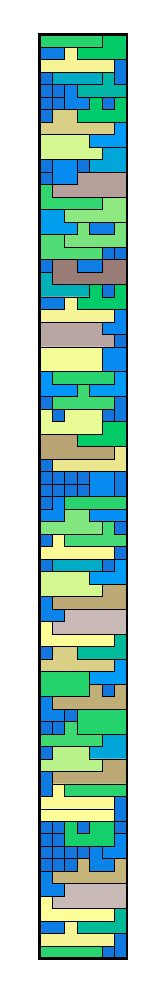}
\end{subfigure}                                       
\begin{subfigure}{.11\textwidth}                      
\centering                                            
  \includegraphics[scale=0.8]{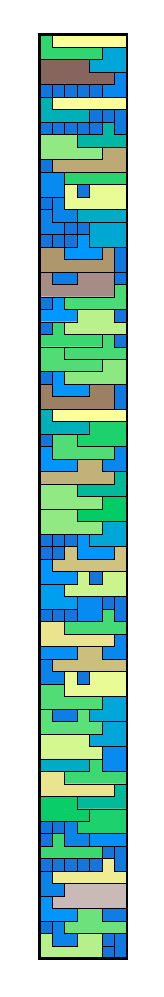}
\end{subfigure}                                       
\begin{subfigure}{.11\textwidth}                      
\centering                                            
  \includegraphics[scale=0.8]{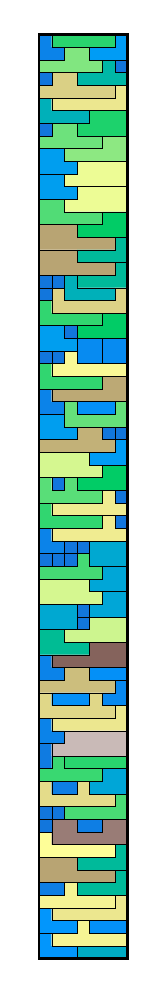}
\end{subfigure}                                       
\begin{subfigure}{.11\textwidth}                      
\centering                                            
  \includegraphics[scale=0.8]{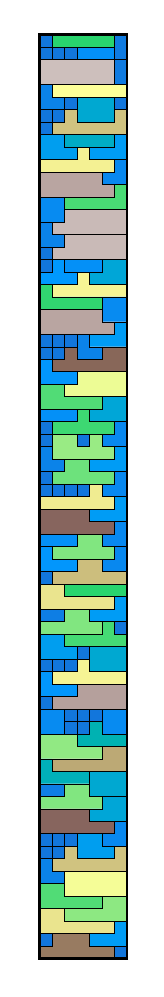}
\end{subfigure}                                       
\begin{subfigure}{.11\textwidth}                      
\centering                                            
  \includegraphics[scale=0.8]{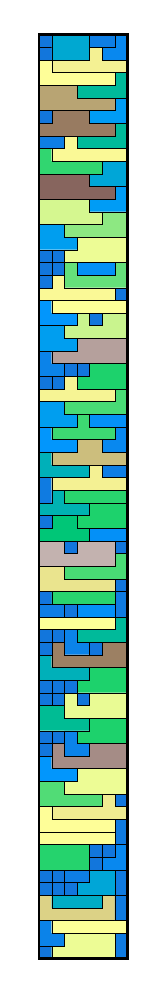}
\end{subfigure}                                       
\begin{subfigure}{.11\textwidth}                      
\centering                                            
  \includegraphics[scale=0.8]{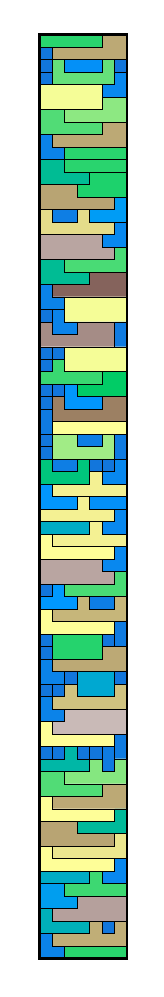}
\end{subfigure}
\caption{Eight random ${n \times 7}$ tilings of areas in the interval
$[500; 520]$ using in total $95$ different tiles.}
  \label{fig:tilings-7-126}
\end{figure}

The automaton corresponding to our tiling sampler consists of more than $2 000$
states and $28, 000$ transitions.
We remark that this example is a notable
improvement over the work of Bodini and Ponty~\cite{BodPonty} who were able to
sample ${n \times 6}$ tilings using $7$ different tiles (we handle $126$) with a corresponding
automaton consisting of roughly $1 500$ states and $3 200$ transitions.

\subsection{Simply-generated trees with node degree constraints.}
Next, we give an example of simple varieties of plane trees with fixed sets of
admissible node degrees, satisfying the general equation $$y(z) = z \phi(y(z))
\quad \text{for some polynomial} \quad \phi \colon \mathbb{C} \to \mathbb{C}\,
.$$

Let us consider the case of plane trees where nodes have degrees in the set $D
= \set{0,\ldots,9}$, i.e.~$\phi(y(z)) = a_0 + a_1 y(z) + a_2 {y(z)}^2 + \cdots
+ a_{9} {y(z)}^{9}$.  Here, the numbers \( a_0, a_1, a_2, \ldots, a_{9} \) are
nonnegative real coefficients.  We tune the corresponding algebraic
specification so to achieve a target frequency of $1\%$ for all nodes of
degrees $d \geq 2$. Frequencies of nodes with degrees $d \leq 1$ are left
undistorted. For presentation purposes all nodes with equal degree are given
the same unique, distinguishable colour. \autoref{fig:tree} depicts two
exemplary trees generated in this manner.
 \begin{figure}[ht!]
  \begin{subfigure}{.32\paperwidth}
      \centering
  \includegraphics[scale=0.07]{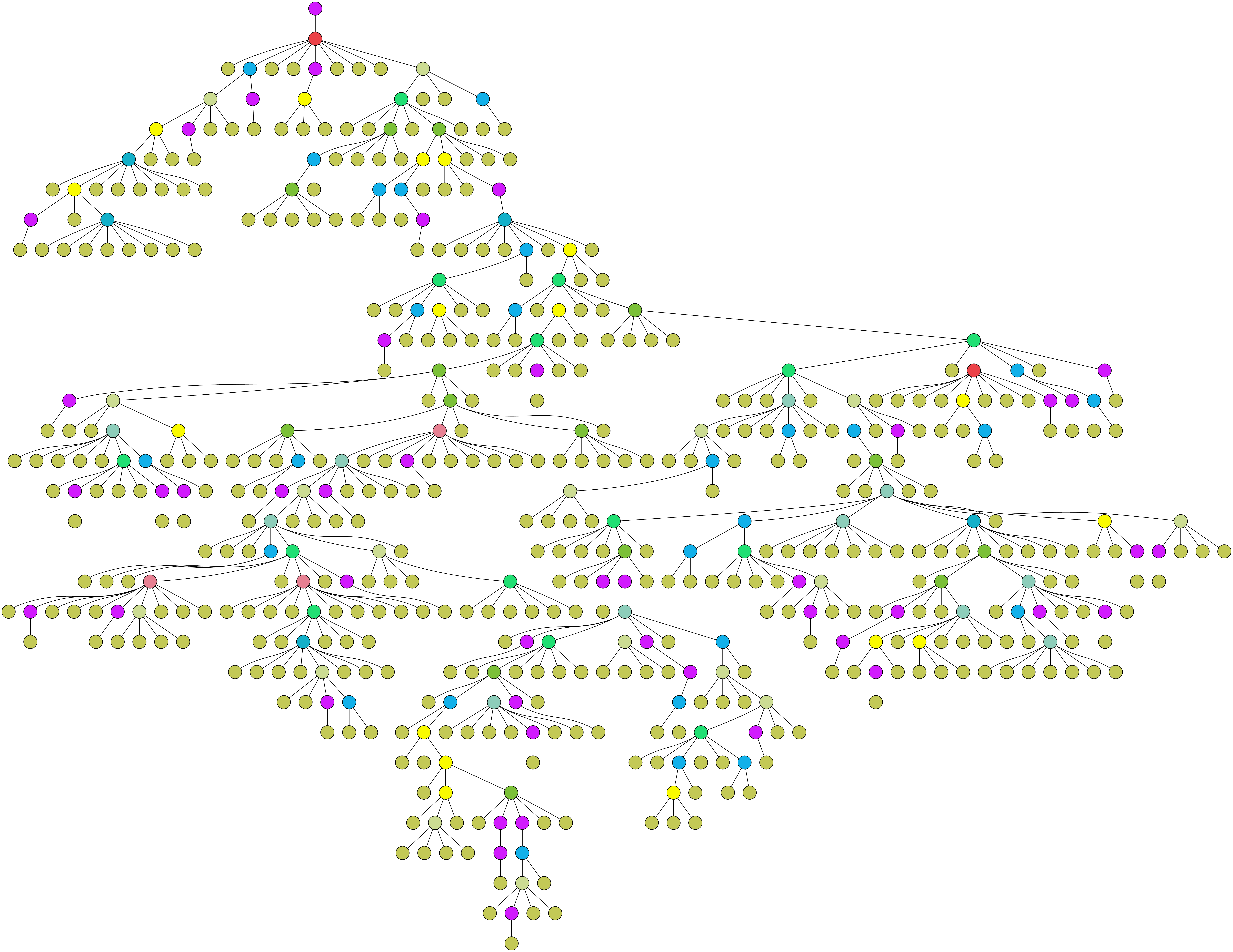}
\end{subfigure}
\begin{subfigure}{.4\paperwidth}
    \centering
  \includegraphics[width=\textwidth]{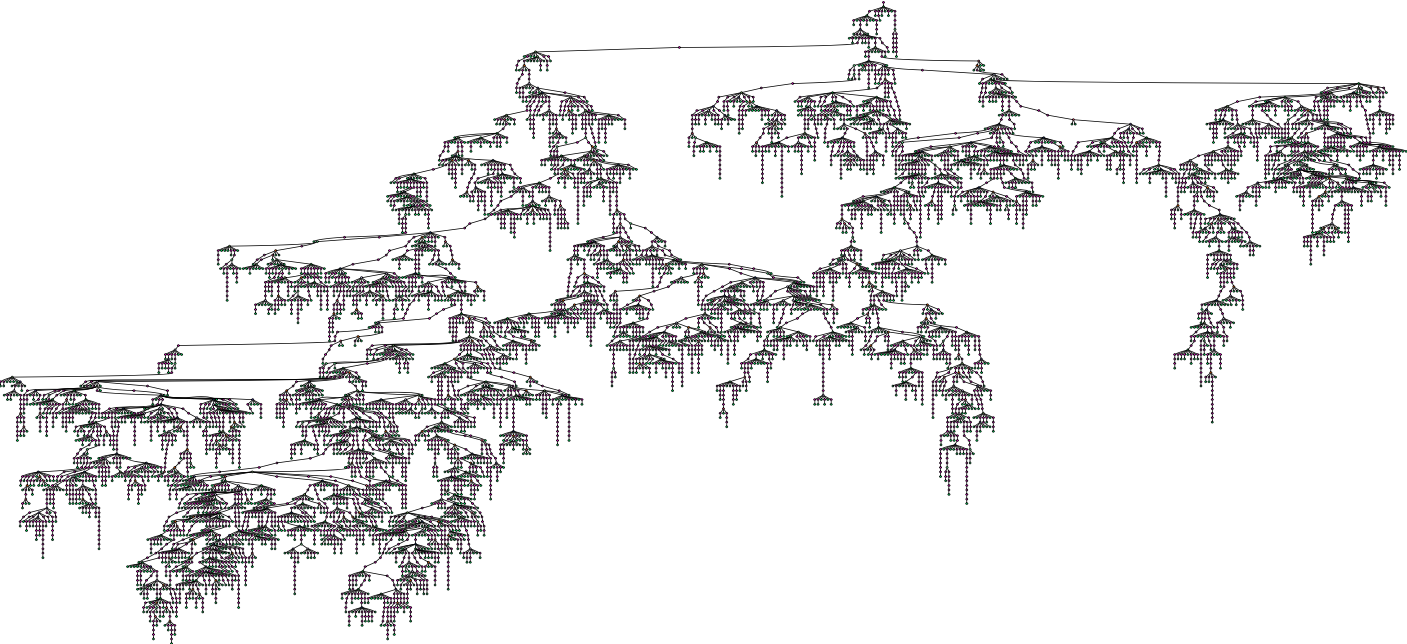}
\end{subfigure}
 \caption{Two random plane trees with degrees in the set ${D =
     \set{0,\ldots,9}}$. On the left, a tree of size in between $500$ and $550$;
     on the right, a tree of size in the interval $[10, 000; 10, 050]$.}
\label{fig:tree}
 \end{figure}

Empirical frequencies for the right tree of~\autoref{fig:tree} and a
simply-generated tree of size in between $10,000$ and $10,050$ with default
node degree frequencies are included in~\autoref{fig:tree-freqs}.

\begin{table*}[ht!]
\scalebox{.8}{
\begin{tabular}{c | c | c | c | c | c | c | c | c | c | c}
Node degree & $0$ & $1$ & $2$ & $3$ & $4$ & $5$ &
$6$ & $7$ & $8$ & $9$ \\ \hline
Tuned frequency & - - - & - - - & $1.00\%$ & $1.00\%$ & $1.00\%$ &
$1.00\%$ & $1.00\%$ & $1.00\%$ & $1.00\%$ & $1.00\%$\\
Observed frequency & $35.925\%$ & $56.168\%$ & $0.928\%$ & $0.898\%$ & $1.098\%$ &
$0.818\%$ & $1.247\%$ & $0.938\%$ & $1.058\%$ & $0.918\%$\\
Default frequency & $50.004\%$ & $24.952\%$ & $12.356\%$ & $6.322\%$ &
$2.882\%$ & $1.984\%$ & $0.877\%$ & $0.378\%$ & $0.169\%$ & $0.069\%$
\end{tabular}}
\caption{Empirical frequencies of the node degree distribution.}
\label{fig:tree-freqs}
\end{table*}

We briefly remark that for this particular problem, Bodini, David and Marchal
proposed a different, bit-optimal sampling procedure for random trees with
given partition of node degrees~\cite{DBLP:conf/caldam/BodiniJM16}.

\subsection{Variable distribution in plain $\lambda$-terms.}
To exhibit the benefits of distorting the intrinsic distribution of various
structural patterns in algebraic data types, we present an example
specification defining so-called plain $\lambda$\nobreakdash-terms with explicit control
over the distribution of de~Bruijn indices.

In their nameless representation due to de~Bruijn~\cite{deBruijn1972}
$\lambda$\nobreakdash-terms are defined by the formal grammar $L ::= \lambda L~|~(L L)~|~D$
where $D = \set{\idx{0},\idx{1},\idx{2},\ldots}$ is an infinite denumerable set
of so-called indices (cf.~\cite{BendkowskiGLZ16,GittenbergerGolebiewskiG16}).
Assuming that we encode de~Bruijn indices as a sequence of successors of zero
(i.e.~use a unary base representation), the class $\mathcal{L}$ of plain
$\lambda$\nobreakdash-terms can be specified as \( \mathcal{L} = \mathcal{Z} \mathcal{L} +
\mathcal{Z} {\mathcal{L}}^2 + \mathcal{D} \) where \( \mathcal{D} = \mathcal{Z}
\Seq(\mathcal{Z}) \).  In order to control the distribution of de~Bruijn
indices we need a more explicit specification for de~Bruijn indices. For
instance: \[ \mathcal{D} = \mathcal{U}_0 \mathcal{Z} + \mathcal{U}_1
    {\mathcal{Z}}^2 + \cdots + \mathcal{U}_k {\mathcal{Z}}^{k+1} +
\mathcal{Z}^{k+2} \Seq(\mathcal{Z})\, . \] Here, we roll out the $k+1$ initial indices and
assign distinct marking variables to each one of them, leaving the remainder
sequence intact. In doing so, we are in a position to construct a sampler tuned
to enforce a uniform distribution of $8\%$ among all marked indices,
i.e.~indices $\idx{0},\idx{1},\ldots,\idx{8}$, distorting in effect their
intrinsic geometric distribution.

\autoref{fig:lambda-term}~illustrates two random $\lambda$\nobreakdash-terms with
such a new distribution of indices. For presentation purposes, each index
in the left picture is given a distinct colour.

 \begin{figure}[ht!]
  \begin{subfigure}{.4\textwidth}
\centering
  \includegraphics[scale=0.1]{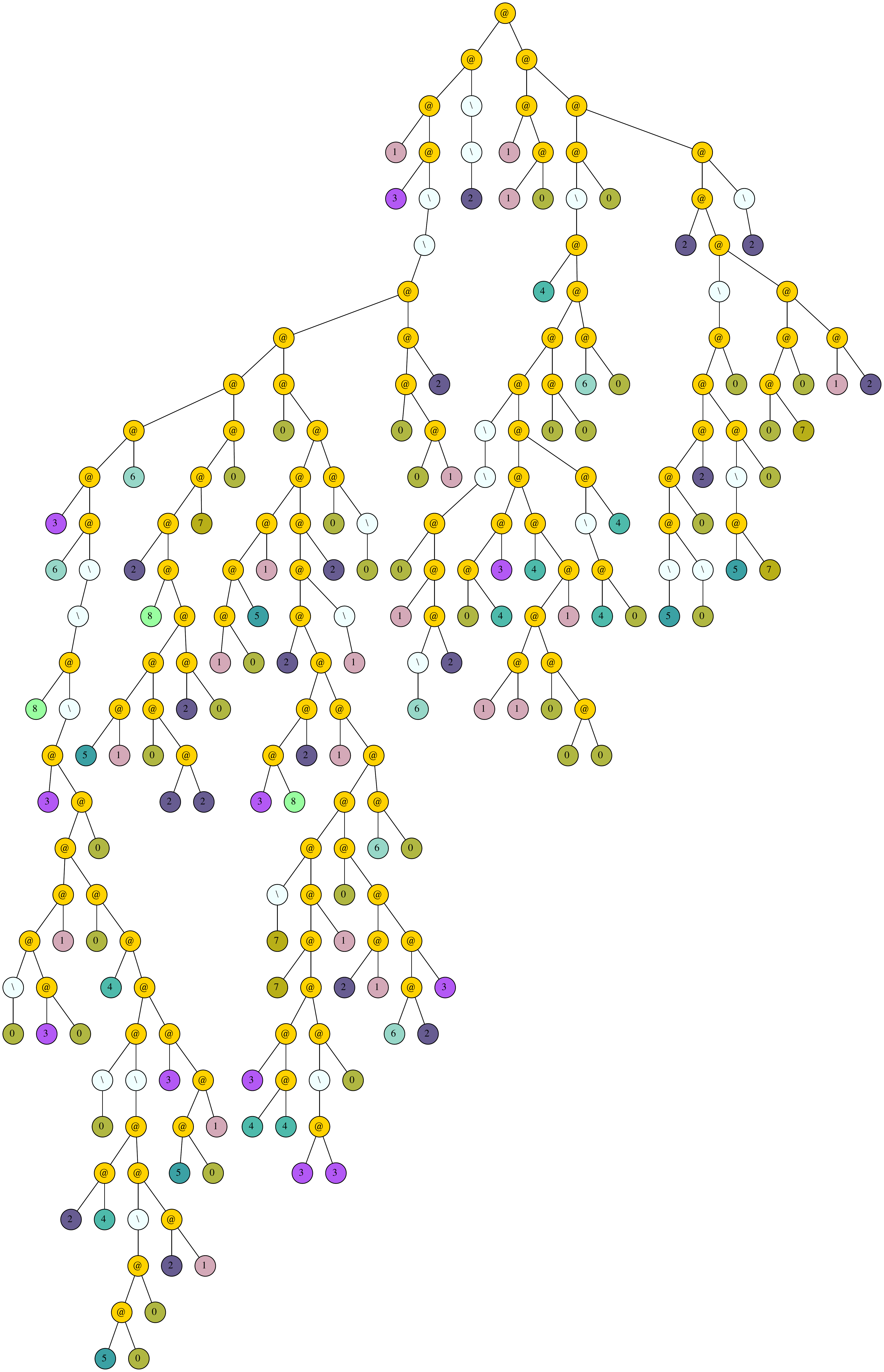}
\end{subfigure}
\begin{subfigure}{.4\textwidth}
\centering
  \includegraphics[width=\textwidth]{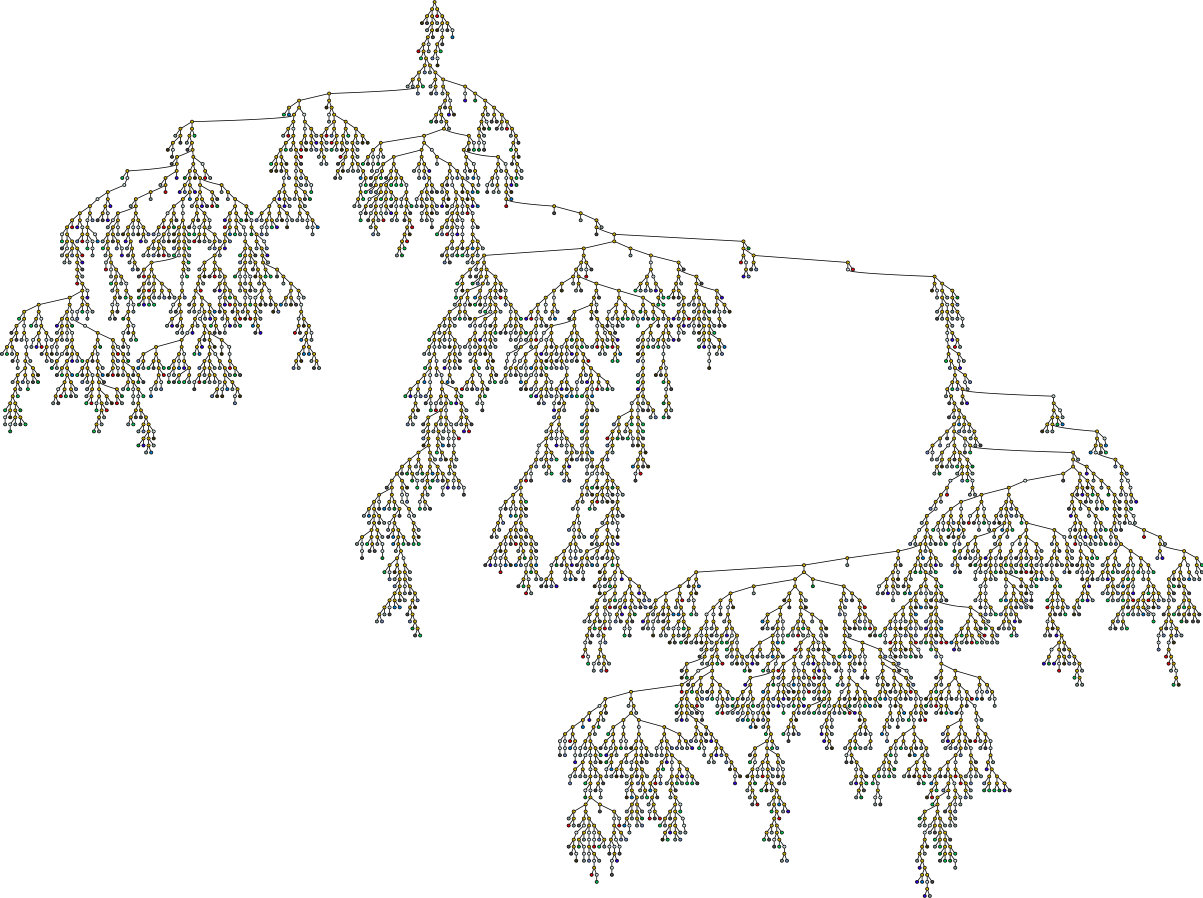}
\end{subfigure}
 \caption{On the left, a random $\lambda$\nobreakdash-term of size in the interval
 $[500;550]$; on the right, a larger example of a random $\lambda$\nobreakdash-term of size between $10,000$ and $10,050$.}
\label{fig:lambda-term}
 \end{figure}

Empirical frequencies for the right term of~\autoref{fig:lambda-term}
and a plain $\lambda$\nobreakdash-term of size in between $10,000$ and $10,050$ with
default de~Bruijn index frequencies are included
in~\autoref{fig:lambda-term-freqs}.

\begin{table*}[ht!]
\scalebox{.85}{
\begin{tabular}{c | c | c | c | c | c | c | c | c | c}
Index & $\idx{0}$ & $\idx{1}$ & $\idx{2}$ & $\idx{3}$ & $\idx{4}$ & $\idx{5}$ &
$\idx{6}$ & $\idx{7}$ & $\idx{8}$\\ \hline
Tuned frequency & $8.00\%$ & $8.00\%$ & $8.00\%$ & $8.00\%$ & $8.00\%$ &
$8.00\%$ & $8.00\%$ & $8.00\%$ & $8.00\%$\\
Observed frequency & $7.50\%$ & $7.77\%$ & $8.00\%$ & $8.23\%$ & $8.04\%$ &
$7.61\%$ & $8.53\%$ & $7.43\%$ & $9.08\%$\\
Default frequency & $21.91\%$ & $12.51\%$ & $5.68\%$ & $2.31\%$ &
$0.74\%$ & $0.17\%$ & $0.20\%$ & $0.07\%$ & - - -
\end{tabular}}
\caption{Empirical frequencies (with respect to the term size) of index distribution.}
\label{fig:lambda-term-freqs}
\end{table*}

Let us note that algebraic data types, an essential conceptual ingredient of
various functional programming languages such as Haskell or OCaml, and the
random generation of their inhabitants satisfying additional structural or
semantic properties is one of the central problems present in the field of
property-based software testing (see, e.g.~\cite{Claessen-2000,palka2012}). In
such an approach to software quality assurance, programmer-declared function
invariants (so-called properties) are checked using random inputs, generated
accordingly to some predetermined, though usually not rigorously controlled,
distribution. In this context, our techniques provide a novel and effective
approach to generating random algebraic data types with fixed average
frequencies of type constructors. In particular, using our methods it is
possible to \emph{boost} the intrinsic frequencies of certain desired subpatterns or
\emph{diminish} those which are unwanted.

\subsection{Weighted partitions.}
\label{subsection:weighted:partitions}
Integer partitions are one of the most intensively studied objects in number
theory, algebraic combinatorics and statistical physics. Hardy and Ramanujan
obtained the famous asymptotics which has later been refined by
Rademacher~\cite[Chapter VIII]{flajolet09}. In his article~\cite{Vershik1996},
Vershik considers several combinatorial examples related to statistical
mechanics and obtains the limit shape for a random integer partition of size \(
n \) with \( \alpha \sqrt{n} \) parts and summands bounded by \( \theta
\sqrt{n} \).
Let us remark that Bernstein, Fahrbach, and
Randall~\cite{bernstein2017analyzing} have recently analysed
the complexity of exact-size Boltzmann sampler for weighted partitions.
% When I wrote ``Recently, Bernstein, Fahrbach, and Randall'' it seemed like
% ``Recently'' is the first author of the four.
In the model of ideal gas, there are several particles (bosons)
which form a so-called assembly of particles. The overall energy of the system
is the sum of the energies
\(
\Lambda = \sum_{i=1}^N \lambda_{\vec i}
\)
where \( \lambda_i \) denotes the energy of \( i \)-th particle.  We assume
that energies are positive integers. Depending on the energy level \( \lambda
\) there are \( j(\lambda) \) possible available states for each particle; the
function \( j(\lambda) \) depends on the physical model. Since all the
particles are indistinguishable, the generating function \( P(z) \) for the
number of assemblies \( p(\Lambda) \) with energy \( \Lambda \) takes the form
\begin{equation}
    P(z) = \sum_{\Lambda=0}^\infty p(\Lambda) z^\Lambda =
        \prod_{\lambda > 0}
        \dfrac{1}{(1 - z^\lambda)^{j(\lambda)}}
    \enspace .
\end{equation}
In the model of \( d \)-dimensional harmonic trap (also known as the
Bose-Einstein condensation) according to \cite{chase1999canonical,
haugset1997bose, lucietti2008asymptotic} the number of states for a particle
with energy \( \lambda \) is \( { d + \lambda - 1 \choose \lambda } \) so that
each state can be represented as a multiset with \( \lambda \) elements having
\( d \) different colours.  Accordingly, an assembly is a multiset of particles
(since they are bosons and hence indistinguishable) therefore the generating
function for the number of assemblies takes the form
\begin{equation}
    P(z) =
    \MSet(\MSet\nolimits_{\geq 1}(\mathcal Z_1 + \cdots + \mathcal Z_d))
    \enspace .
\end{equation}

It is possible to control the expected frequencies of colours using our tuning
procedure and sample resulting assemblies as Young tableaux. Each row
corresponds to a particle whereas the colouring of the row displays the
multiset of included colours, see~\autoref{fig:bose:einstein}. We also
generated weighted partitions of expected size \( 1000 \) (which are too large
to display) with tuned frequencies of \( 5 \) colours,
see~\autoref{table:partition:frequencies}.

 \begin{figure}[ht!]
\begin{subfigure}{0.2\textwidth}
\centering
    \includegraphics[height=0.2\textheight]{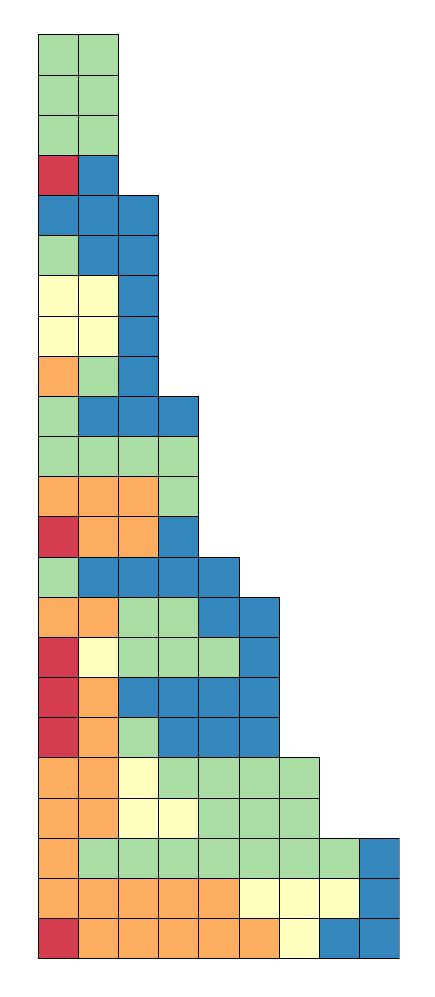}
    \caption{\tiny [5, 10, 15, 20, 25]}                           
\end{subfigure}                                                   
\begin{subfigure}{0.25\textwidth}                                 
\centering                                                        
    \includegraphics[height=0.2\textheight]{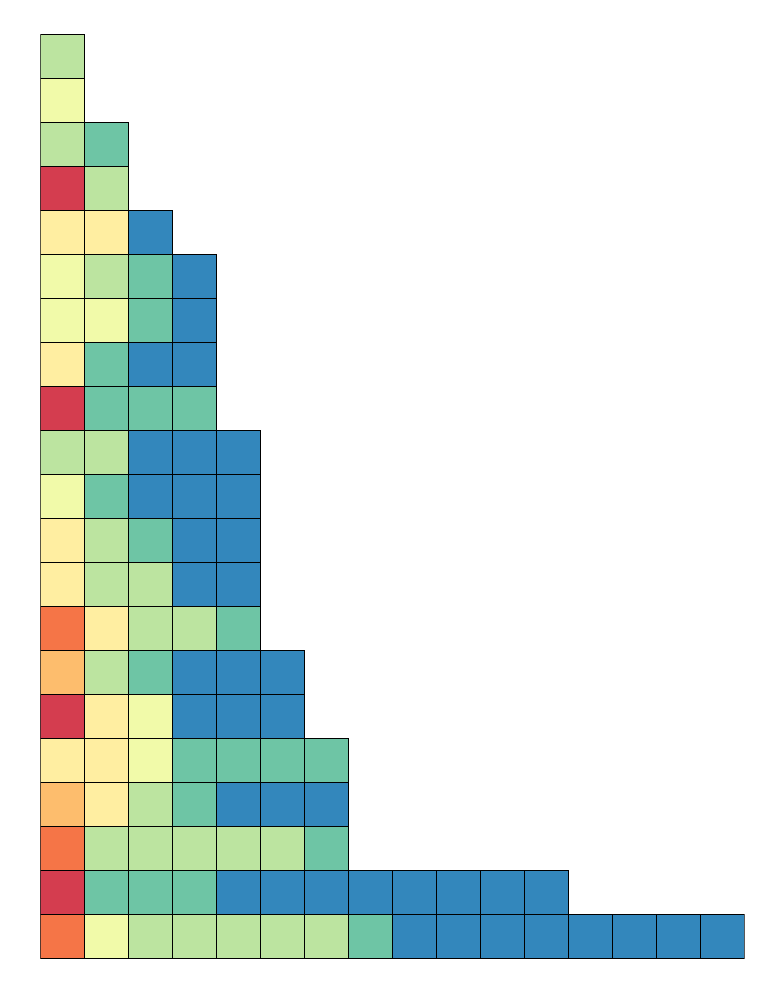}
    \caption{\tiny [4,4,4,4, 10, 20, 30, 40]}                     
\end{subfigure}                                                   
\begin{subfigure}{0.25\textwidth}                                 
\centering                                                        
    \includegraphics[height=0.2\textheight]{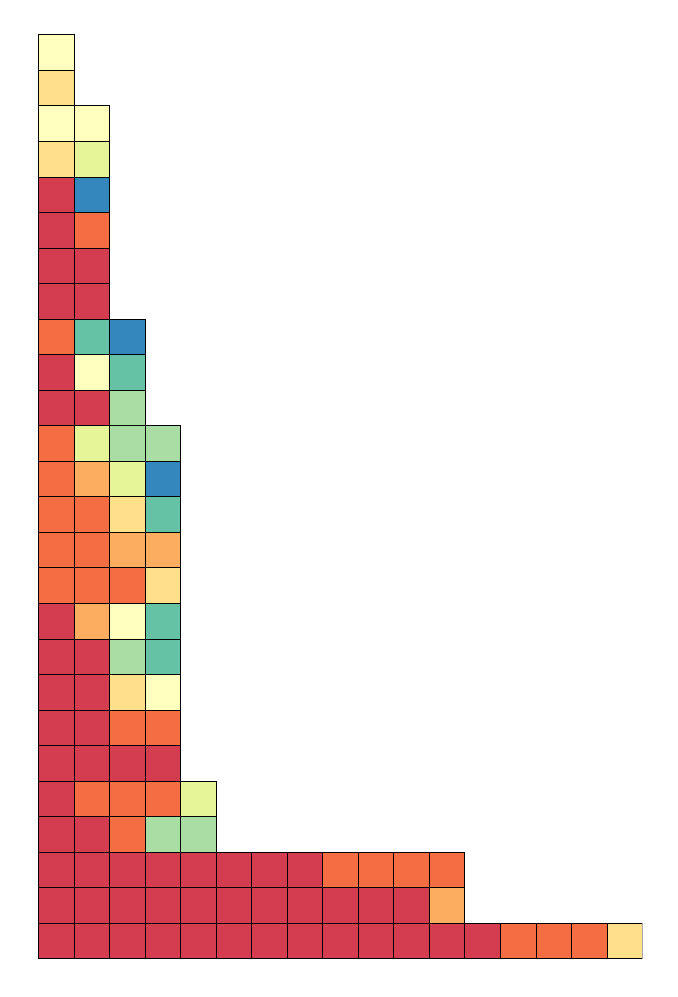}
    \caption{\tiny [80, 40, 20, 10, 9, 8, 7, 6, 5]}
\end{subfigure}                                                   
\begin{subfigure}{0.25\textwidth}                                 
\centering                                                        
    \includegraphics[height=0.2\textheight]{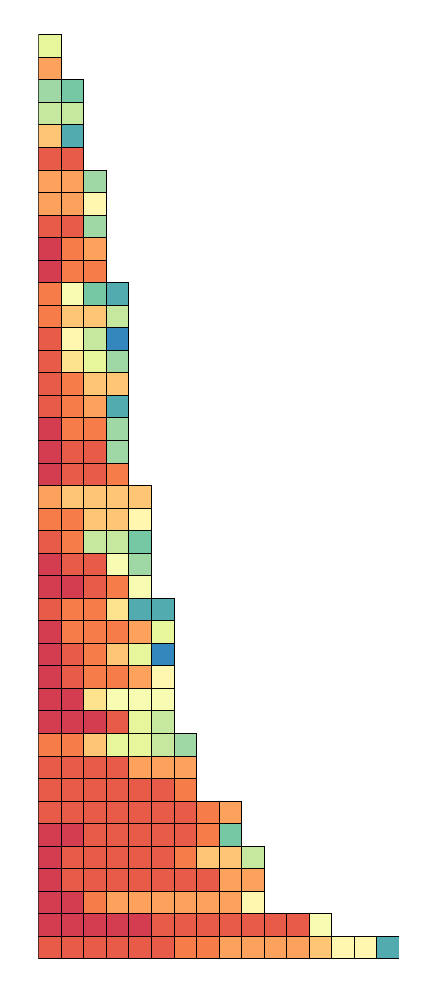}
    \caption{\tiny [20, 60, 30, 20, 10, $5^9$]}
\end{subfigure}
    \caption{Young tableaux corresponding to Bose--Einstein condensates with
        expected numbers of different colours. Notation \( [c_1, c_2, \ldots,
        c_k] \) provides the expected number \( c_j \) of the \( j \)-th
        colour, \( c_k^m \) is a shortcut for \( m \) occurrences of \( c_k \).}
    \label{fig:bose:einstein}
\end{figure}

\begin{table}[ht!]
\scalebox{.8}{
\begin{tabular}{c | c | c | c | c | c | c}
Colour index & $\idx{1}$ & $\idx{2}$ & $\idx{3}$ & $\idx{4}$ & $\idx{5}$ & size\\ \hline
    Tuned frequency    & $0.03$ & $0.07$ & $0.1$ & $0.3$ & $0.5$ & 1000\\ \hline
                       & $0.03$ & $0.08$ & $0.07$ & $0.33$ & $0.49$ & 957\\
                       & $0.03$ & $0.06$ & $0.09$ & $0.28$ & $0.54$ & 1099\\
    Observed frequency & $0.03$ & $0.08$ & $0.09$ & $0.34$ & $0.46$ & 992\\
                       & $0.04$ & $0.07$ & $0.1$  & $0.31$ & $0.49$ & 932\\
                       & $0.04$ & $0.09$ & $0.1$  & $0.25$ & $0.52$ & 1067\\
\end{tabular}}
\caption{Empirical frequencies of colours observed in random partition.}
\label{table:partition:frequencies}
\end{table}

Let us briefly explain our generation procedure. Boltzmann sampling for the outer \(
\MSet \) operator is described
in~\autoref{algorithm:mset},~\autoref{section:polya:structres}. The sampling of
inner \( \MSet_{\geq 1}(\mathcal Z_1 + \ldots + \mathcal Z_d) \) is more
delicate. The generating function for this multiset can be written as
\begin{equation}
    \MSet\nolimits_{\geq 1}( z_1 + \cdots + z_d ) = \prod_{i=1}^d \dfrac{1}{1 - z_i} - 1
    \enspace .
\end{equation}
In order to correctly calculate the branching probabilities, we introduce slack
variables \( s_1, \ldots, s_d \) satisfying \( (1 + s_i) = (1 - z_i)^{-1}
\). Boltzmann samplers for the newly determined combinatorial classes \( \Gamma
\mathcal S_i \) are essentially Boltzmann samplers for \( \Seq_{\geq
1}(\mathcal Z_i) \). Let us note that after expanding brackets the expression
becomes
\begin{multline*}
    \MSet\nolimits_{\geq 1} (z_1 + \cdots + z_d) ={}
    (s_1 + \cdots + s_d) 
     + (s_1 s_2 + \cdots + s_{d-1} s_d)
     + \cdots + s_1 s_2 \ldots s_d.
\end{multline*}
The total number of summands is \( 2^{d} - 1 \) where each summand corresponds
to choosing some subset of colours. Finally, let us explain how to precompute
all the symmetric polynomials and efficiently handle the branching process in
quadratic time using a dynamic programming approach. We can recursively define
two arrays of real numbers \( p_{k,j} \) and \( q_{k,j} \) satisfying
\begin{equation}
    \begin{cases}
        p_{1,j} = s_j , \quad j \in  \{ 1, \ldots, d \}; \\
        q_{k,d} = p_{k,d}, \quad k \in \{ 1, \ldots, d \}; \\
        q_{k,j} = p_{k,j} + q_{k,j+1},
            \quad j \in \{ k, \ldots, d-1 \},
            \quad k \in \{ 1, \ldots, d \};  \\
        p_{k,j} = s_{j-k+1} \cdot q_{k-1,j},
            \quad j \in \{ k, \ldots, d-1 \},
            \quad k \in \{ 2, \ldots d \};
    \end{cases}
\end{equation}
Arrays \( (p_{k,j})_{j=k}^d \) contain the branching probabilities determining
the next colour inside the \( k \)-th symmetric polynomial. Arrays \(
(q_{k,j})_{j=k}^d \) contain partial sums for the \( k \)-th symmetric
polynomial and are required in intermediate steps. Numbers \( q_{k,k} \) are
equal to the total values of symmetric polynomials \( (s_1 + \cdots + s_2),
(s_1s_2+\cdots +s_{d-1}s_d), \ldots, s_1s_2,\ldots, s_d \) and they define
initial branching probabilities to choose the number of colours.

\newpage
\subsection{Prototype sampler generator.}
Consider the following example of an input file for {\sf Boltzmann Brain}:

\begin{lstlisting}[mathescape,
                numbersep=5pt,
                gobble=2,
                frame=lines,
                framesep=2mm,
                numbers=none]
  -- Motzkin trees
  Motzkin = Leaf (3)
           | Unary Motzkin
           | Binary Motzkin Motzkin (2) [0.3].
\end{lstlisting}

\newcommand{\hs}[1]{{\color{blackred}\texttt{#1}}}

Here, a \hs{Motzkin} algebraic data type is defined. It consists of three
constructors: a constant \hs{Leaf} of weight three, a \hs{Unary} constructor of
weight one (default value if not explicitly annotated) and a constructor
\hs{Binary} of weight two together with an explicit tuning frequency of $30\%$.
Such a definition corresponds to the combinatorial specification $ \mathcal{M}
= {\mathcal{Z}}^3 + \mathcal{Z} \mathcal{M} + \mathcal{U} {\mathcal{Z}}^2
{\mathcal{M}}^2$ where the objective is to obtain the mean proportion of
$\mathcal{U} {\mathcal{Z}}^2 {\mathcal{M}}^2$ equal $30\%$ of the total
structure size. All the terms \hs{Leaf}, \hs{Unary}, \hs{Motzkin}, \hs{Binary}
are user-defined keywords. Given such a specification on input, {\sf bb} builds
a corresponding singular Boltzmann sampler implemented in form of a self-contained Haskell module.

\printbibliography

\appendix

\section{Convex optimisation: proofs and algorithms}
\label{section:convex:proofs}

Until now, we have left several important questions unanswered.  Firstly, what
is the required precision \( \varepsilon \) for multiparametric tuning?
Secondly, what is its precise computational complexity? In order to determine
the time and space complexity of our tuning procedure we need to explain some
technical decisions regarding the choice of particular optimisation methods.
In this section we prove that the optimisation procedures described
in \S~\ref{section:tuning} give the correct solution to the tuning problem.

\subsection{Proofs of the theorems.}
\hfill\\[2mm]
\noindent
\textit{Proof of~Theorem~\ref{theorem:general}.}
Let the following \emph{nabla-notation} denote the vector of derivatives (so-called
gradient vector) with
respect to the variable vector \( \vec z = (z_1, \ldots, z_k) \):
\begin{equation}
    \nabla_{\vec z} f(\vec z) =
    \left(
        \dfrac{\partial}{\partial z_1} f(\vec z),
        \ldots,
        \dfrac{\partial}{\partial z_k} f(\vec z)
    \right)^\top\, .
\end{equation}
We start with noticing that tuning the expected number of atom occurrences
is equivalent to solving the equation
\(
    \nabla_{\vec \xi} \log C(e^{\vec \xi}) = \vec \nu
\), see~Proposition~\ref{proposition:expected:value}.
Here, the right-hand side is equal to \( \nabla_{\vec \xi} (\vec \nu^\top \vec \xi) \) so
tuning is further equivalent to
\(
    \nabla_{\vec \xi} \left(
        \log C(e^{\vec \xi}) - \vec \nu^\top \vec \xi
    \right) = 0
\).
The function under the gradient is convex as it is a sum of a convex and
linear function. In consequence, the problem of minimising the function is equivalent to
finding the root of the derivative
\begin{equation}
    \log C(e^{\vec \xi}) - \vec \nu^\top \vec \xi \to \min_{\vec \xi}
    \enspace .
\end{equation}
\begin{Definition}{(Feasible points)}
    In the optimisation problem
    \begin{equation}
        \begin{cases}
            f(\vec z) \to \min, \\
            \vec z \in \Omega
        \end{cases}
    \end{equation}
    a point \( \vec z \) is called feasible if it belongs to the set \(
    \Omega \).
\end{Definition}

\vspace{2mm}\noindent
\textit{Proof of~Theorem~\ref{theorem:algebraic:tuning}.}
Let \( \vec N = (N_1, \ldots, N_k) \) be the vector of atom occurrences of each type.
    Consider the vector \( \vec z^\ast \) such that
    \(
    \mathbb E_{\vec z^\ast} (\vec N) = \vec \nu.
    \)
Let \( \vec c \) denote the logarithms of the values of generating functions at
point \( \vec z^\ast = e^{\vec \xi^\ast} \). Clearly, in such a case all inequalities in~\eqref{eq:optimisation:algebraic} become equalities and the point \( (\vec c,
\vec \xi^\ast) \) is feasible.

Let us show that if the point
\( (\vec c, \vec \xi ) \) is optimal, then all the inequalities in~\eqref{eq:optimisation:algebraic} become equalities.
Firstly, suppose that the inequality
\begin{equation}
    c_1 \geq \log \Phi_1(e^{\vec c}, e^{\vec \xi})
\end{equation}
    does not turn to an equality. Certainly, there is a \emph{gap} and the value \( c_1 \) can be
decreased. In doing so, the target function value is decreased as well. Hence,
the point \( (\vec c, \vec \xi) \) cannot be optimal.

Now, suppose that the initial inequality does turn to equality,
however \( c_k > \log \Phi_k(e^{\vec c}, e^{\vec \xi}) \) for some $k \neq 1$.
Since the system is strongly connected, there exists a path
\( P = c_1 \to c_2 \to \cdots \to c_k \) (indices are
chosen without loss of generality) in the
corresponding dependency graph. Note that for pairs of consecutive
variables $(c_i, c_{i+1})$ in $P$, the function $\log \Phi_i(e^{\vec c}, e^{\vec \xi})$
is strictly monotonic in $c_{i+1}$ (as its monotonic and references $c_{i+1}$).
In such a case we can decrease $c_{i+1}$ so to assure that
$c_i > \log \Phi_i(e^{\vec c}, e^{\vec \xi})$ while the point
\( (\vec c, \vec \xi) \) remains feasible. Decreasing
$c_{i+1},c_i,\ldots,c_1$ in order, we finally arrive at a
feasible point with a decreased target function value. In consequence,
\( (\vec c, \vec \xi) \) could not have been optimal to begin with.

So, eventually, the optimisation problem reduces to minimising the expression
subject to the system of equations
\(
    \vec c = \log \vec \Phi(e^{\vec c}, e^{\vec \xi})
\)
or, equivalently,
\(
    \vec C(\vec z) = \vec \Phi(\vec C(\vec z), \vec z)
\)
and can be therefore further reduced to~Theorem~\ref{theorem:general}.

\vspace{2mm}\noindent
\textit{Proof of~Theorem~\ref{theorem:singular:tuning}.}
By similar reasoning as in the previous proof, we can show that the maximum is attained when all
the inequalities turn to equalities. Indeed, suppose that at least one inequality
is strict, say
\(
    c_j > \log \Phi_j(e^{\vec c}, e^{\xi}, e^{\vec \eta})
\).
The value \( c_j \) can be slightly decreased by \( \varepsilon \) by choosing
    a sufficiently small distortion \( \varepsilon \) to turn all the
    equalities containing \( c_j \) in the right-hand side \( \log
    \Phi_i(e^{\vec c}, e^{\xi}, e^{\vec \eta}) \) to strict inequalities,
    because the right-hand sides of each of the inequalities are monotonic
    functions with respect to \( c_j \).  This procedure can be repeated until
    all the equalities turn into inequalities. Finally, we slightly decrease
    the value \( \xi \) to increase the target function while still staying
    inside the feasible set, because of the monotonicity of
    the right-hand side with respect to \( \xi \).

Let us fix \( \vec u = e^{\vec \eta} \). For rational and algebraic grammars, within the
Drmota--Lalley--Woods framework, see for instance~\cite{drmota1997systems}, the
 corresponding generating function singular approximation takes the form
\begin{equation}
\label{eq:asymptotic:singular}
    C(z, \vec u) \sim a_0(\vec u) - b_0(\vec u) \left(1 - \dfrac{z}{\rho(\vec
    u)}\right)^{t}
    \enspace .
\end{equation}
If \( t < 0 \), then the asymptotically dominant term becomes \( -b_0 \left( 1 - \frac{z}{\rho(\vec u)} \right)^{t}\). In this case, tuning the target expected frequencies
corresponds to solving the following equation as $z \to \rho(u)$:
\begin{equation}
\label{eq:asymptotic:expectation}
\mathrm{diag}(\vec u)
    \dfrac
    { [z^n] \nabla_{\vec u} C(z, \vec u) }
    { [z^n] C(z, \vec u) } = n \vec \alpha \enspace .
\end{equation}
Let us substitute the asymptotic expansion~\eqref{eq:asymptotic:singular}
into~\eqref{eq:asymptotic:expectation} to track how \( \vec u \) depends on \(
\vec \alpha \):
\begin{equation}
\mathrm{diag}(\vec u)
    \dfrac
        {
            [z^n]
            t b_0(\vec u)
            \left(
                1 - \dfrac{z}{\rho(\vec u)}
            \right)^{t-1}
            z \dfrac{\nabla_{\vec u}\rho(\vec u)}{\rho^2(\vec u)}
        }
        {
            [z^n]
            b_0(\vec u)
            \left(
                1 - \dfrac{z}{\rho(\vec u)}
            \right)^{t}
        }
        = - n \vec \alpha
    \enspace .
\end{equation}
Only dominant terms are accounted for. Then, by the binomial theorem
\begin{equation}
\mathrm{diag}(\vec u)
 b_0(\vec u)
 \dfrac{t}{n} { t-1 \choose n }
    \dfrac{z\nabla_{\vec u}\rho(\vec u)}{\rho^2(\vec u)}
    b_0(\vec u)^{-1}
    { t \choose n}^{-1}
    = - \vec \alpha
    \enspace ,
\end{equation}
With \( z = \rho(\vec u) \), as \( n \to \infty \),
we obtain
after cancellations
\begin{equation}
    \mathrm{diag}(\vec u)
    \dfrac
        {\nabla_{\vec u} \rho(\vec u)}
        {\rho(\vec u)}
    = - \vec \alpha
\end{equation}
which can be rewritten as
\begin{equation}\label{eq:asymptotic:expectation:two}
    \nabla_{\vec \eta} \log \rho( e^{\vec \eta} ) = -\vec \alpha
    \enspace .
\end{equation}
Passing to exponential variables~\eqref{eq:asymptotic:expectation:two} becomes
\begin{equation}
    \nabla_{\vec \eta} ( \xi(\vec \eta) + \vec \alpha^\top \vec \eta ) = 0
    \enspace .
\end{equation}
As we already discovered, the dependence \( \xi(\vec \eta) \) is given by the
system of equations because the maximum is achieved only when all inequalities
turn to equations. That is, tuning the singular sampler is equivalent to
maximising \( \xi + \vec \alpha^\top \vec \eta \) over the set of feasible
points.

\begin{Remark}
    For ordinary and singular samplers, the corresponding feasible set remains
    the same; what differs is the optimised target function. Singular samplers
    correspond to imposing an infinite target size. In practice, however, the
    required singularity is almost never known \emph{exactly} but rather
    calculated up to some feasible finite precision. The tuned structure size
    is therefore enormously large, but still, nevertheless, finite. In this
    context, singular samplers provide a natural \emph{limiting} understanding of
    the tuning phenomenon and as such, there are several possible ways of
    proving~Theorem~\ref{theorem:singular:tuning}.
\end{Remark}

\autoref{fig:binary:picture} illustrates the feasible set for the class of
binary trees and its transition after applying the log-exp transform, turning
the set into a convex collection of feasible points.  In both figures, the
singular point is the rightmost point on the plot. Ordinary sampler tuning
corresponds to finding the tangent line which touches the set, given the angle
between the line and the abscissa axis.
\begin{figure}[hbt!]
    \begin{center}
        \includegraphics[width=0.49\textwidth]{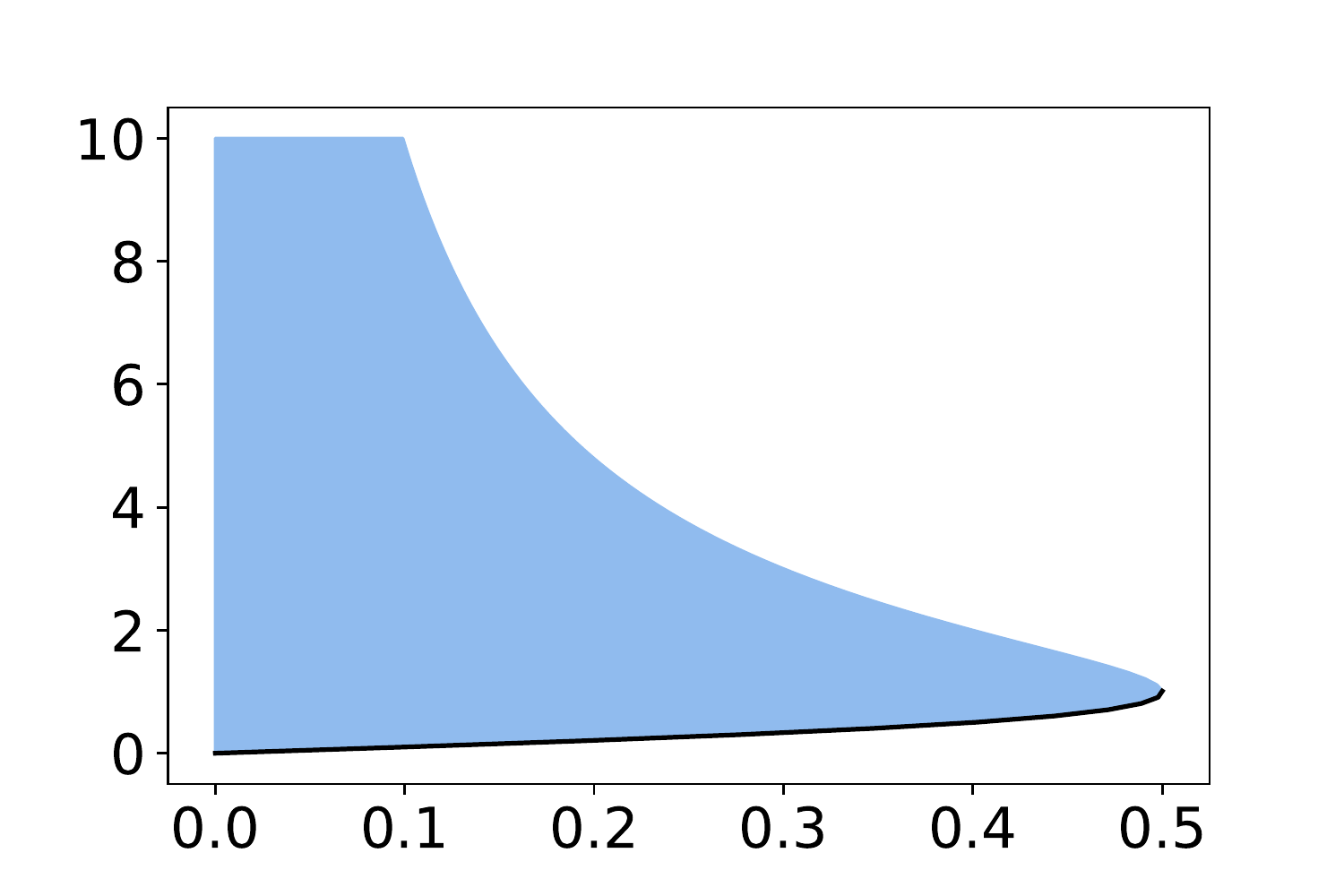}
        \includegraphics[width=0.49\textwidth]{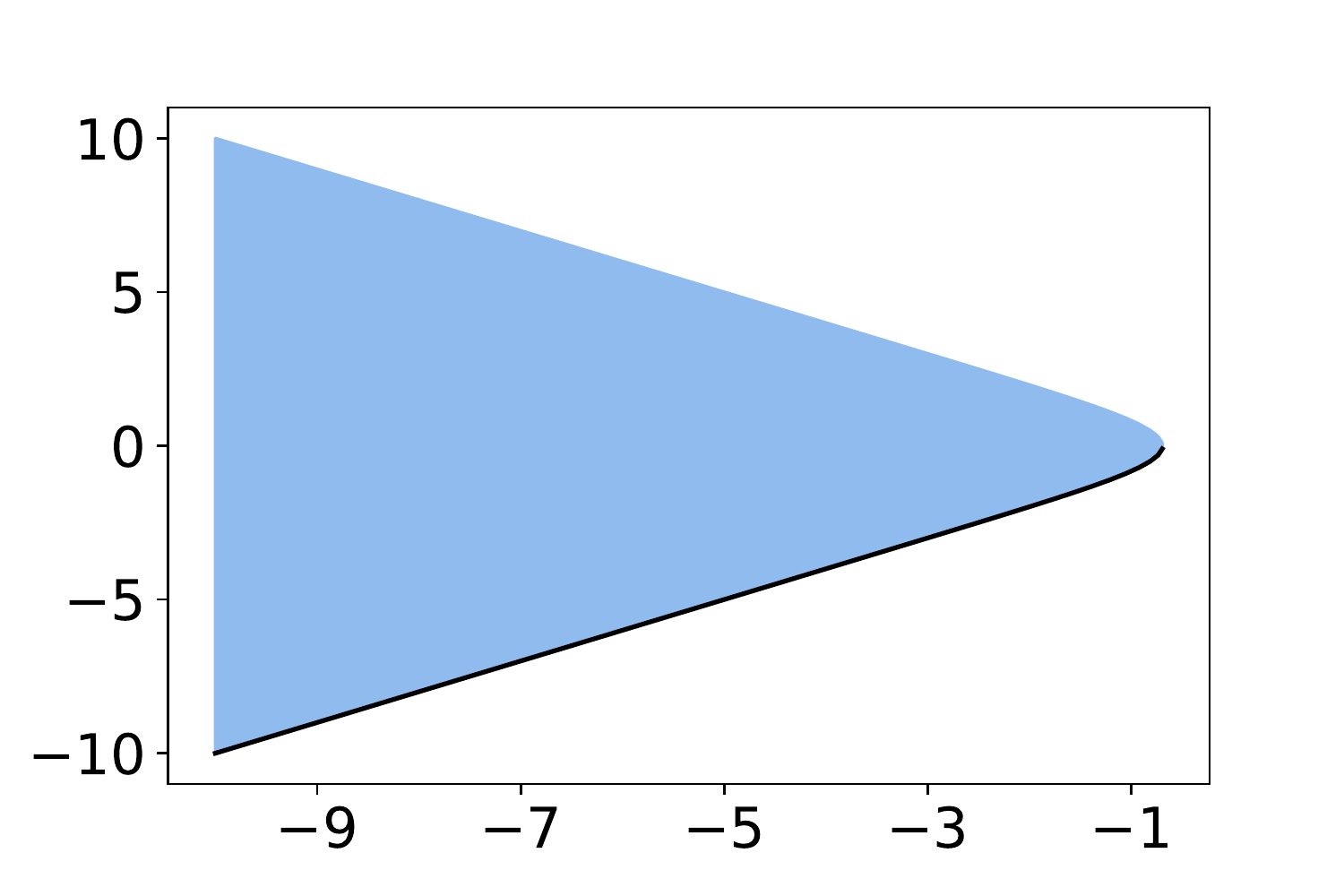}
    \end{center}
    \caption{Binary trees \( B \geq z + zB^2 \) and log-exp transform of the
    feasible set. The black curve denotes the principal branch of the generating
function $B(z)$ corresponding to the class of binary trees.}
    \label{fig:binary:picture}
\end{figure}

\subsection{Disciplined convex programming and optimisation algorithms.}
In the subsequent proofs, we present the framework of Disciplined Convex
Programming (DCP in short) and show how to incorporate elementary combinatorial
constructions into this framework.
Nesterov and Nemirovskii \cite{nesterov1994interior} developed a seminal polynomial-time
optimisation algorithm for convex programming which involves the construction of certain
\emph{self-concordant barriers} related to the feasible set of points.
The arithmetic complexity of their method is
\begin{equation}
    O\left(
        \log \dfrac{1}{\varepsilon}
        \sqrt{\vartheta}
        \mathcal N
    \right)
\end{equation}
where \( \mathcal N \) is the arithmetic complexity of a single Newton iteration step,
\( \vartheta \) is the so-called constant of self-concordeness of the barriers
and \( \varepsilon \) is the target precision. Before we go into each of the
terms, we mention that for sparse matrix representations, it is possible to
accelerate the speed of the Newton iteration, i.e.~the step of solving the system of
linear equations
\begin{equation}
    A \vec x = \vec b
    \quad \text{where} \quad
    A \in \mathbb R^{m \times m} \text{ and }
    \vec x, \vec b \in \mathbb R^{m}
\end{equation}
from \( O(m^3) \) to \( O(m^2) \).

Unfortunately, for general convex programming problems there is no constructive
general-purpose barrier construction method, merely existence proofs.
Fortunately, Grant, Boyd, and Ye~\cite{grant2006disciplined} developed
the DCP framework which
automatically constructs suitable barriers for the user. Moreover, DCP
also automatically provides the starting feasible point which is itself a
nontrivial problem in general.  As its price, the user is obliged to provide a
certificate that the constructed problem is convex, i.e.~express all convex
functions in terms of a predefined set of elementary convex functions.

In our implementation, we rely on two particular solvers, a second-order
(i.e.~using second-order derivatives) Embedded Conic Solver
(\texttt{ECOS})~\cite{domahidi2013ecos} and recently developed first-order
(i.e.~using only first-order derivatives) Splitting Conic Solver (\texttt{SCS})
algorithm~\cite{o2016conic}. The conversion of the DCP problem into
its standard form is done using  {\sf cvxpy}, a Python-embedded modelling
language for disciplined convex programming~\cite{cvxpy}.

\vspace{2mm}\noindent
\textit{Proof of~Theorem~\ref{theorem:main}.}
We start with showing that the tuning procedure can be effectively
represented in the framework of DCP.

In our case, every inequality takes the form
\begin{equation}\label{eq:dcp:inequality}
    c_i \geq \log \left(
        \sum_{i=1}^m e^{\ell_i (\vec c, \vec z)}
    \right)
\end{equation}
where \( \ell_i(\vec c, \vec z) \) are some linear functions.
Appreciably, the log-sum-exp function belongs to the set of admissible constructions
of the DCP framework.

Converting the tuning problem into DCP involves creating some slack variables.
For each product of two terms \( X \times Y \) we create slack
variables for \( X \) and \( Y \) which are represented by
the variables \( \xi \) and \( \eta \) in the log-exp realm as
\begin{equation}
    e^{\xi} = X \qquad \text{and} \qquad e^{\eta} = Y
    \enspace .
\end{equation}
Next, we replace \( X \times Y \) by \( e^{\xi + \eta}  \) as composition of
    addition and exponentiation is a valid DCP program. Since every expression
    in systems corresponding to considered combinatorial classes is a sum of
    products, the corresponding restriction~\eqref{eq:dcp:inequality} is
    converted to a valid DCP constraint using the elementary log-sum-exp
    function.

The sequence operator \( \Seq(\mathcal A) \) which converts a generating function
    \( A(\vec z) \) into \( (1 - A(\vec z))^{-1} \) is \emph{unfolded} by adding an
extra equation into the system in form of
\begin{equation}
    D := \Seq A(\vec z)  \quad \text{whereas} \quad
    D = 1 + A D\, .
\end{equation}
Two additional constructions, \( \MSet \) and \( \Cycle \) are treated in a
    similar way. Infinite sums are replaced by finite ones because the
    difference in the distribution of truncated variables is a negative
    exponent in the truncation length, and hence negligible.

Using the DCP method, the constant of self-concordness of the barriers is equal
    to \( \vartheta = O(L) \), where \( L \) is the length of the problem
    description. This includes the number of combinatorial classes, number of
    atoms for which we control the frequency and the sum of lengths of
    descriptions of each specification, i.e.~their overall length.  In total,
    the complexity of optimisation can be therefore crudely estimated as
\begin{equation}
    O \left(
        L^{3.5} \log \frac{1}{\varepsilon}
    \right) \, .
\end{equation}
Certainly, the complexity of tuning is polynomial, as stated. We emphasise that in
practice, using sparse matrices this can be further reduced to \(
O(L^{2.5} \log (1/\varepsilon)) \).

\begin{Remark}
    Weighted partitions, one of our previous applications, involves a multiset operator
    \( \MSet\nolimits_{\geq 1}(\CS Z_1 + \cdots + \CS Z_d) \) which generalises
    to
    \(
    \Seq(\mathcal C_1) \Seq(\mathcal C_2) \cdots \Seq(\CS C_d) - 1
    \)
    and does not immediately fall into the category of admissible
    operators as it involves subtraction. This is a
    general \emph{weak point} of Boltzmann sampling involving usually a huge
    amount of rejections, in consequence substantially slowing down the generation
    process. Moreover, it also disables our convex optimisation tuning procedure
    because the constructions involving the minus sign cease to be convex
    and therefore do not fit the DCP framework.

    We present
    the following change of variables for this operator, involving a quadratic number of slack
    variables.
    The \( \Seq(\CS C_i) \) operator yielding the generating function \( (1 -
    C_i(\vec z))^{-1} \) is replaced by \( (1 + \CS S_i) \) where \( \CS S_i \)
    satisfies
    \begin{equation}
        \CS S_i = \CS C_i + \CS S_i \CS C_i
        \enspace .
    \end{equation}
    Next, we expand all of the brackets in the product \( \prod_{i=1}^d(1 + \CS S_i) - 1 \).
    Consequently, we define the following arrays \( \CS P_{i,j} \) and \( \CS Q_{i,j} \):
    \begin{equation}
    \begin{cases}
        \CS P_{1,j} = \CS C_j, \quad j \in \{1, \ldots, d \} \\
        \CS Q_{k,d} = \CS P_{k,d}, \quad k \in \{ 1, \ldots, d \} \\
        \CS Q_{k,j} = \CS P_{k,j} + \CS Q_{k,j+1},
            \quad
            j \in \{ k, \ldots, d-1\},\
            k \in \{ 1, \ldots, d \} \\
        \CS P_{k,j} = \CS C_{j-k+1} \cdot \CS Q_{k-1,j},
            \quad
            j \in \{ k, \ldots, d-1\},\
            k \in \{ 2, \ldots, d \}
        \enspace .
    \end{cases}
    \end{equation}
    Semantically, as in \S~\ref{subsection:weighted:partitions},
    \( (\mathcal P_{i,j})_{j=k}^d \) and \( (\mathcal Q_{i,j})_{j=k}^d \)
    denote the summands inside symmetric polynomials
    \begin{equation}
        \begin{cases}
            \CS Q_{1,1} = \CS S_{1} + \CS S_2 + \ldots + \CS S_d \enspace ,\\
            \CS Q_{2,2} = \CS S_1 ( \CS S_2 + \ldots + \CS S_d) +
                          \CS S_2 ( \CS S_3 + \ldots + \CS S_d) +
                          \ldots + \CS S_{d-1} \CS S_d \enspace , \\
                          \CS Q_{3,3} = \CS S_1 (
                                \CS S_2 \CS S_3 +
                                \ldots +
                                \CS S_{d-1} \CS S_d) +
                          \ldots +
                          \CS S_{d-2} \CS S_{d-1} \CS S_d
        \end{cases}
    \end{equation}
    and the auxiliary partial sums used to recompute the consequent
    expressions, respectively. So for instance when \( d = 5 \)
    we obtain
    \def\sa{{\CS S_1}}
    \def\sb{{\CS S_2}}
    \def\sd{{\CS S_3}}
    \def\se{{\CS S_4}}
    \def\sh{{\CS S_5}}
    \begin{equation}
        \mathcal P =
        \begin{pmatrix}
        \sa    &  \sb                 & \sd         & \se                 &  \sh
            \\
        0      &  \sa(\sb+\ldots+\sh) & \ldots      & \sd(\se+\sh)        &  \se \sh
            \\
        0      &  0                   & \sa(\ldots) & \sb(\sd(\se+\sh)+\se \sh)
                                                                &  \sd \se \sh
            \\
        0      &  0                   & 0           & \sa(\ldots)
                                                                &  \sb \ldots \sh
            \\
        0      &  0                   & 0           & 0  &  \sa \ldots \sh
        \end{pmatrix}\, .
    \end{equation}
    The union of classes in each row gives corresponding symmetric polynomial
    \( \CS Q_{k,k} \),
    and the partial sum of elements in the row gives the elements of \( \CS Q \).

    Finally, the expression
    \( \prod_{i=1}^d (1 + \CS S_i) - 1 \)
    is replaced by the sum of elementary symmetric polynomials
    \(
        \CS Q_{1,1} + \CS Q_{2,2} + \cdots + \CS Q_{d,d}
    \)
    where we have (combinatorially)
    \begin{equation}
        \CS Q_{j,j} = \sum_{1 \leq i_1 < \ldots < i_j \leq d} \CS S_{i_1} \cdots \CS S_{i_j}
        \enspace .
    \end{equation}

    We emphasise that the last sum is not meant to be implemented in practice
    in a na\"{i}ve way as it would take an exponential amount of time to be
    computed.
\end{Remark}

\subsection{Tuning precision.} 
In this section, we only work with algebraic systems that 
meet the certain regularity conditions from Drmota--Lalley--Woods
Theorem~\cite{drmota1997systems}.

%TODO: some objects in this section are not clearly defined

%In the algebraic case, Drmota--Lalley--Woods
%Theorem~\cite{drmota1997systems},\cite[Theorem VII.6]{flajolet09}, states that
%under certain mild conditions
%the generating functions (or components) entering 
%the corresponding polynomial system
%have the same radius of convergence \( \rho < \infty \), and 
%for each component there exists an analytic function \( h \) such that in a
%neighbourhood of \( \rho \), the component can be represented by 
%\( h(\sqrt{1 - z/\rho}) \).

\begin{proposition}
Consider a multiparametric combinatorial specification
\begin{equation}
    \vec {\mathcal Y} = \vec \Phi(\vec {\mathcal Y}, \mathcal Z, 
    \vec{\mathcal U})    
\end{equation}
whose corresponding system of equations
is either rational or algebraic.
Suppose that we sample objects
from the class \( \mathcal F = \mathcal Y_1 \) with target 
expected sizes \( (n, \nu_1 n, \ldots, \nu_d n) \),
where \( \nu_i \) are constants, \( n \to \infty \).
Let \( F(z, \vec u) \) be the multivariate generating function corresponding 
to the class \( \mathcal F \),
and let \( (z^\ast, \vec u^\ast) \) be the target tuning vector. 
Then, there exists \( \varepsilon = \Theta(1/Poly(n)) \)
such that the points \( (z, \vec u) \) from
the \( \varepsilon \)-ball centered at \( (z^\ast, \vec u^\ast) \) intersected
with the set of feasible points
\begin{equation*}
    \Big\{
        (z, \vec u) \in \mathbb R^{1+d} \mid
        \vec Y(z, \vec u) \geq \vec \Phi(\vec Y(z, \vec u), z, \vec u),
        \
        \| (z^\ast - z, \vec u^\ast - \vec u) \| \leq \varepsilon
    \Big\}
\end{equation*}
yield expectations within \( O(1) \) of target expectations:
\begin{eqnarray*}
    z F'_z(z, \vec u) / F(z, \vec u) &=& n + O(1)
    \enspace ,
    \\
    u_i F'_{u_i}(z, \vec u) / F(z, \vec u) &=& \nu_i n + O(1)
    ,
    \
    i\in\{1, \ldots, d\}
    \enspace .
\end{eqnarray*}
\end{proposition}

\begin{proof}
Let us show that \( z^\ast \), as a function of \( n \), satisfies
\begin{equation}
    \label{eq:target:tuning}
\begin{cases}    
    z^\ast(n) \sim \rho(1 - \alpha / n) ,
    & \mathcal F \text{ is rational;}
    \\
    z^\ast(n) \sim \rho(1 - C / n^2) ,
    & \mathcal F \text{ is algebraic.}
\end{cases}
\end{equation}
Here, \( \alpha \) is a positive integer depending on the rational system, \( C
\) is a generic constant. We also note that the same asymptotics is valid
for each coordinate of the vector \( \vec u^\ast(n) \), up to multiplicative
constants depending on \( \nu_i \) and the values of \( \alpha \) and \( C \).

For rational systems, there exist analytic functions
\( \beta(z, \vec u) \), \( \rho(\vec u) \)
and a positive integer \( \alpha \)
such that
\begin{equation}
    \label{eq:rational:asymptotic:expansion}
    F(z, \vec u) \sim \beta(z, \vec u) (1 - z/\rho(\vec u))^{-\alpha}    
    \enspace , \quad
    z \to \rho(\vec u)
    \enspace .
\end{equation}
After substituting the asymptotic expansion~\eqref{eq:rational:asymptotic:expansion}
into~\eqref{eq:expectations},
we obtain the first part of~\eqref{eq:target:tuning}. 

For algebraic systems, according to 
Drmota--Lalley--Woods Theorem~\cite{drmota1997systems},
there exist analytic functions
\( \alpha(z, \vec u), \beta(z, \vec u), \rho(\vec u) \) such that
as \( z \to \rho(\vec u) \),
\begin{equation}
    \label{eq:algebraic:asymptotic:expansion}
    F(z, \vec u) \sim \alpha(z, \vec u) - 
    \beta(z, \vec u) (1 - z / \rho(\vec u))^{1/2}
    \enspace .
\end{equation}
Again, substituting this asymptotic expansion into~\eqref{eq:expectations},
we obtain
\begin{equation}
    z^\ast (n) \dfrac{\beta}{2 \rho \alpha}
    \left(
        1 - \dfrac{z^\ast(n)}{\rho}
    \right)^{-1/2} \sim n
    \enspace .
\end{equation}
Taking into account that \( z^\ast(n) = \rho+o(1) \), this implies the second
part of~\eqref{eq:target:tuning}. Similarly, this can be applied to each
coordinate of \( \vec u \), not only to \( z \).
                                    
Let us handle the tuning precision. We use the mean value theorem to bound \(
\varepsilon \). Let \( m = n + O(1) \). Then,
\begin{equation*}
    \varepsilon^2 \geq \| (z^\ast(n)-z^\ast(m),
        \vec u^\ast(n)-\vec u^\ast(m)) \|^2
    =
    (z^\ast(n) - z^\ast(m)^2 + 
    \sum_{i=1}^d (u_i^\ast(n) - u_i^\ast(m))^2
    \enspace .
\end{equation*}
By the mean value theorem, there exist numbers
\( (n'_i)_{i=0}^d \) from the interval \( [n,m] \) such that
\begin{eqnarray*}
    z^\ast(n) - z^\ast(m) &=& (n-m) \dfrac{d z^\ast}{dn}(n'_0)\ ,
    \\
    u_i^\ast(n) - u_i^\ast(m) &=&
        (n-m) \dfrac{d u_i^\ast}{dn} (n'_i) ,\
    i \in \{ 1, \ldots, d \}\ .
\end{eqnarray*}
Thus, as \( n-m = O(1) \), we obtain
\[
    \varepsilon^2 \geq
    O(1) \left[
        \left(
            \dfrac{d z^\ast}{dn} (n'_0)
        \right)^2
        + 
        \sum_{i=1}^d
        \left(
            \dfrac{d u_i^\ast}{dn}(n'_i)
        \right)^2
    \right]
    \enspace .
\]
Since \( n'_i = n + O(1) \),
after substituting~\eqref{eq:target:tuning} and expressing the derivatives,
we obtain the bound \( {\varepsilon = O(n^{-2})} \) for rational grammars and
\( {\varepsilon = O(n^{-3})} \) for algebraic specifications.
\end{proof}
\begin{Remark}
    If one uses the \emph{anticipated rejection} principle for sampling the
    objects of approximate size \( n + O(1) \), in effect rejecting objects
    smaller than \( n - O(1) \) and ``killing'' the generation of objects whose
    size exceeds \( n + O(1) \),
    it is possible to have a more relaxed bound \( \varepsilon = O(n^{-2}) \)
    for the case of algebraic specifications. Even though the expected size of
    generated objects will be smaller than \( n \), so that we will need a large
    number of restarts, the total amount of generated atoms will be
    nevertheless linear in \( n \).
    We refer to~\cite[Theorem 4.1]{BodGenRo2015} for further discussion. 
\end{Remark}
\begin{Remark}
    Under an extra frequency rejection (independently of the structure size)
    \emph{it is not possible}
    to get rid of the assumption of strong connectivity and get a general
    estimate on the complexity of rejection-based sampling for arbitrary
    combinatorial specifications.
    Let us recall that Banderier, Bodini, Ponty and Bouzid give combinatorial classes
    with non-continuous parameter distributions~\cite{banderier2012diversity}.
    For instance, consider the combinatorial class
\begin{equation}
    \CS F = \Seq ( \CS Z^3 )         \Seq ( \CS U   \CS Z^3) +
            \Seq ( \CS U^2 \CS Z^3 ) \Seq ( \CS U^3 \CS Z^3)
\end{equation}
in which all the structures have parameter frequencies in the intervals \( (0,
    \tfrac13) \) and \( (\tfrac23, 1) \). Certainly, tuning the sampler for a
    target frequency inside the interval \( (\tfrac13, \tfrac23) \) yields a
    rejection sampler which never stops as there is no structures of demanded
    frequency.

For this reason we restrict our attention on two important subclasses of
    combinatorial specifications, i.e.~strongly connected rational and
    algebraic languages.  Due to Bender and Richmond~\cite{BenderCLT} both
    classes follow a multivariate Gaussian law with linear expectation and
    standard deviation.  In consequence, corresponding multiparametric
    Boltzmann samplers work in linear time if we accept a linear tolerance for
    the size $[(1-\epsilon)n,(1+\epsilon)n]$ and a square root tolerance for
    the parameters $[f-\kappa/\sqrt{n},f+\kappa/\sqrt{n}]$.
\end{Remark}

\section{Samplers for rational grammars}
\label{section:rational:grammars}
Recall that a strongly connected rational grammar
\begin{equation}
    \label{eq:rational:grammar}
    \vec F = \vec \Phi(\vec F, \vec z)
\end{equation}
is a specification corresponding to a rational language
 whose dependency graph is strongly connected. State and transitions
 of the associated automaton correspond to classes \(\vec F = (F_1, \ldots, F_m)\)
 and to appropriate monomials in
the system~\eqref{eq:rational:grammar}, respectively.

For rational samplers, we decide to implement the strategy of
\emph{interruptible sampling}, introduced in~\cite{bacher2013exact} as the
so-called \emph{Hand of God} principle.  The idea of anticipated rejection is
also discussed in~\cite{BodGenRo2015}.  We start with fixing two distinguished
states of the automaton, a starting one and a final (terminal) one. The
starting and the terminal states may coincide. Next, we construct a tuned
variant of the corresponding singular sampler. Specifically, tune it with
arbitrarily high, yet still feasible precision. In essence, it is enough to
tune to expected quantities exceeding the target ones.  While tuning, we add a
constraint \( \| \vec v \| \leq M \) where \( \vec v \) contains all the
variables \( \vec F \) and \( \vec z \), and \( M \) is a logarithm of a large
number, say \( M = 40 \). This constraint is required because otherwise the
value of associated generating functions tends to infinity as \( \vec z \)
approaches the singular point. Moreover, by doing so we will compute branching
probabilities with an error no more than \( O(e^{-M}) \).  Under these
conditions, the resulting sampler is unlikely to stop with output size less
than the target one.  Finally, we run the sampler from its initial state and
continue sampling until the target structure size is attained.  From that
moment on, we wait for the sampler to naturally reach its final state at which
point the process is interrupted.  In the following proposition we show that
such a sampling procedure is actually an efficient generation scheme.

\begin{proposition}
\label{proposition:interruptible}
	Let $n$ be the target size of an interruptible sampler \(\Gamma \mathcal{S}\) associated
	with a strongly connected rational system \(\mathcal{S}\). Then, the following assertions hold:
	\begin{enumerate}
        \item\label{item:interruptible:a} structures are sampled from a uniform, conditioned on the (composition) size,
		distribution;
    \item\label{item:interruptible:b} the size of the generated structures is \(
        n + O(1) \) in probability where the constant error term depends solely on $\mathcal{S}$.
	\end{enumerate}
\end{proposition}

\begin{proof}
	Let $\omega$ be a structure generated by the interruptible sampler \(\Gamma \mathcal{S}\).
	Assume w.l.o.g. that \(\mathcal{S} =\seq{S_1,\ldots,S_m} \) and moreover $S_1$ and $S_m$
	correspond to the associated automaton's starting and final state, respectively.
	We split the proof into two parts.

    Firstly, let us focus on the uniformity~\eqref{item:interruptible:a}.
    We show that conditioned on the vector of
    quantities \( \vec n \), the probability of a structure \( \omega \)
    with given number of atomic classes is proportional to \( \vec z^{\vec n} \).
	According to the underlying Boltzmann model, each transition $S_i \to S_j$
	taken by \(\Gamma \mathcal{S}\) happens with probability
	 \begin{equation}\label{eq:interruptible:transition}
        \mathbb P_{S_i \to S_j} = \vec z^{\Delta \vec n} \dfrac{S_j(\vec z)}{S_i(\vec
        z)}
    \end{equation}
    where \( \Delta \vec n \) denotes the change in the size of $\omega$
    following transition $S_i \to S_j$.

    Note however that while we trace the interruptible sampler
    generating $\omega$, the ratios of generating functions in~\eqref{eq:interruptible:transition} cancel out (with the exception of the final
     $S_m(\vec{z})$).
     In consequence, the
    probability \(\mathbb P_\omega\) that \(\Gamma \mathcal{S}\) generated
    the structure $\omega$
    becomes
    \begin{equation}
        \mathbb P_{\omega} =
            \vec z^{\sum \Delta \vec n}
            S_m(\vec z) = \vec z^{\vec n}
            S_m(\vec z)
    \end{equation}
    where the latter equality follows from the fact that
    the sum $\sum \Delta \vec n$ of the increments in size is equal to the final
    size $\vec n$. And so, if we
    condition on the composite size, i.e.~the vector of numbers of atoms, the
    distribution is indeed uniform.

    Let us turn to assertion~\eqref{item:interruptible:b}.
    Once the sampler passes the target size,
    it becomes a Markov chain with a single
    absorbing state $S_m$. The chain is irreducible, as the associated
    system is strongly connected, whereas all of the states $S_1,\ldots,S_{m-1}$ are
    not absorbing. Moreover, we can assume that once the target size is reached, the
    sampler starts a random walk in state $S_i$ where $i \neq m$ as otherwise
    our claim holds trivially.

    In consequence, the expected excess outcome size is proportional
    to the expected absorption time starting in the transient state $S_i$.
    This time, however, is known to be finite, see~\cite[Chapter III]{kemeny1960finite}.
	In conclusion, the expected outcome excess size is
	 necessarily finite. An application of Markov's inequality finishes the proof.
\end{proof}

\section{Sampling P\'{o}lya structures}
\label{section:polya:structres}
In the sequel, $X$ denotes the generating function of
$\mathcal{X}$, \( \epsilon \) is an empty sequence.
We present the algorithms from \cite{flajolet2007boltzmann} in order to make the
paper more self-contained.

\begin{algorithm}[ht!]
%\algsetup{linenosize=\tiny}
  %\scriptsize
\caption{$\Gamma \Cycle(\mathcal{A})(\boldsymbol{z})$ }
\label{algorithm:cycle}
\begin{algorithmic}[1]

\REQUIRE{Parameters $\boldsymbol{z}$.}
    \ENSURE{A cycle $\Cycle(\mathcal{A})$.}

    \STATE{Let $K$ be a random variable  in $\mathbb{Z}_{>0}$ satisfying\\
\(
    \mathbb{P}(K = k)
    =
    - \frac{1}{F_{Cyc(\mathcal{A})}}
    \frac {\varphi(k)}{k} \
    \ln\left(
    1-A(\boldsymbol{z}^k)
    \right).
    \)}
\STATE Draw $k$ following the law of $K$.

    \STATE{Let $L$ be a random variable  in $\mathbb{Z}_{>0}$ satisfying\\
\(
    \mathbb{P}(L= \ell)
    =
    -\frac{( A(\boldsymbol{z}^k))^{\ell}}{\ell}
    \frac{1}{\ln(1-A(\boldsymbol{z}^k))}.
    \)}
    \STATE{Draw $\ell$ following the law of $L$,
    $M\gets\epsilon$.}

\FOR
{
    $i$ \emph{\textbf{from}} $1$ \textbf{\emph{to}} $\ell$
}
    \STATE $A_i\gets \Gamma\CS A(\boldsymbol{z}^k)$
    \STATE $M\gets M \cdot A_i$
\ENDFOR
\RETURN  $[M \ldots M]_{k \textrm{ times}}$
\end{algorithmic}
\end{algorithm}
\begin{algorithm}[ht!]
%\algsetup{linenosize=\tiny}
  %\scriptsize
\caption{$\Gamma \MSet(\mathcal{A})(\boldsymbol{z})$}
\label{algorithm:mset}
\begin{algorithmic}[1]

\REQUIRE{Parameters $\boldsymbol{z}$.}
    \ENSURE{A multi-set $\MSet(\mathcal{A})$.}

    \STATE {Let $K$ be a random variable in $\mathbb{Z}_{\geq 0}$ satisfying\\
\(
    \mathbb{P}(K \leq k)
    =
    \prod\limits_{j>k}
    \exp\left(
        -\frac{1}{j}
        A(\boldsymbol{z}^j)
    \right).
    \)}
    \STATE {Draw $k$  following the law of $K$, \(
    S \gets \epsilon
    \).}
    \IF
    {
        \( k > 0 \)
    }

    \FOR
    {
        $j$ \emph{\textbf{from}} $1$ \textbf{\emph{to}} $k-1$
    }
    \STATE{Draw
            \(
            q
            \sim
            \mathrm{Poiss}\big(
                \frac{1}{j}
                A(\vec z^j)
            \big)
    \).
    }
        \FOR
        {
            \( i \) \emph{\textbf{from}} \( 1 \) \textbf{\emph{to}} \( q \)
        }
            \STATE \( A_i\gets \text{$j$ copies of } \Gamma\mathcal{A}(\boldsymbol{z}^j) \)
            \STATE \( S \gets S \cdot A_i \)
        \ENDFOR

    \ENDFOR
    \STATE{Draw \(
            q
            \sim
            \mathrm{Poiss}_{\geq 1}
            \left(
                \frac{1}{k}
                A(\vec z^k)
            \right)
    \).
    }
    \FOR{
        $i$ \emph{\textbf{from}} $1$ \textbf{\emph{to}} $q$
    }
        \STATE $A_i\gets \text{$k$ copies of } \Gamma\mathcal{A}(\boldsymbol{z}^k)$
        \STATE $S\gets S \cdot A_i$\\
    \ENDFOR
    \RETURN $ S $
\ENDIF
\end{algorithmic}
\end{algorithm}

\end{document}